\documentclass{ejm}%[referee]

\usepackage{graphicx,enumerate,amssymb,bbold}
\usepackage{mathrsfs}
\usepackage[notrig]{physics}

\usepackage[colorlinks=true, pdfstartview=FitV, linkcolor=blue,
            citecolor=blue, urlcolor=blue]{hyperref}

\numberwithin{equation}{section}
\numberwithin{figure}{section}
\newtheorem{theorem}{Theorem}[section]
\newtheorem{lemma}[theorem]{Lemma}
\newtheorem{proposition}[theorem]{Proposition}

\newtheorem{remark}[theorem]{Remark}

\newcommand{\bitem}{\begin{itemize}}
\newcommand{\eitem}{\end{itemize}}
\newcommand{\mc}[1]{\mathcal{#1}}
\newcommand{\ms}[1]{\mathscr{#1}}
\newcommand{\mb}[1]{\mathbb{#1}}

\newcommand{\N}{\mathbb{N}}
\newcommand{\R}{\mathbb{R}}

\newcommand{\bpm}{\begin{pmatrix}}
\newcommand{\epm}{\end{pmatrix}}
\newcommand{\bvm}{\begin{vmatrix}}
\newcommand{\evm}{\end{vmatrix}}
\newcommand{\bsm}{\left(\begin{smallmatrix}}
\newcommand{\esm}{\end{smallmatrix}\right)}
\newcommand{\T}{\top}

\newcommand{\ol}[1]{\overline{#1}}
\newcommand{\la}{\langle}
\newcommand{\ra}{\rangle}

\newcommand{\veps}{\varepsilon}

\newcommand{\w}{\omega}

\newcommand{\gdw}{\Leftrightarrow}

\newcommand{\eins}{\mathbb{1}}
\newcommand{\BS}{{\eins_{\mc{S}}}}
\newcommand{\BW}{{\eins_{\mc{W}}}}
\newcommand{\PT}{{\Pi_{0}}}
\newcommand{\RO}{{R}}

\DeclareMathSymbol{\mydiv}{\mathbin}{symbols}{"04}

\DeclareMathOperator{\Diag}{Diag}

\DeclareMathOperator{\ggrad}{grad}
\DeclareMathOperator{\Hess}{Hess}
\DeclareMathOperator{\Egrad}{\partial}
\DeclareMathOperator{\Rgrad}{\ggrad}

\DeclareMathOperator{\Exp}{Exp}

\DeclareMathOperator{\Vol}{Vol}

\title[Continuous-Domain Assignment Flow]{Continuous-Domain Assignment Flows}

\author[F. Savarino and C. Schn\"{o}rr]{%
  F.\ns S\ls A\ls V\ls A\ls R\ls I\ls N\ls O\ns %$\,^1$,\ns
  \and
  C.\ns S\ls C\ls H\ls N\ls \"{O}\ls R\ls R %$\,^2$\ns
}

\affiliation{
  Image and Pattern Analysis Group, Heidelberg University, Heidelberg, Germany\\
  email\textup{\nocorr: \texttt{fabrizio.savarino@iwr.uni-heidelberg.de, schnoerr@math.uni-heidelberg.de}}\\
  URL\textup{\nocorr: \texttt{https://ipa.math.uni-heidelberg.de}}
}

\pagerange{\pageref{firstpage}--\pageref{lastpage}}
\begin{document}

\label{firstpage}
\maketitle

\vspace{-0.25cm}
\begin{abstract}
Assignment flows denote a class of dynamical models for contextual data labeling (classification) on graphs. We derive a novel parametrization of assignment flows that reveals how the underlying information geometry induces two processes for assignment regularization and for gradually enforcing unambiguous decisions, respectively, that seamlessly interact when solving for the flow. Our result enables  to characterize the dominant part of the assignment flow as a Riemannian gradient flow with respect to the underlying information geometry. We consider a continuous-domain formulation of the corresponding potential and develop a novel algorithm in terms of solving a sequence of linear elliptic PDEs subject to a simple convex constraint. Our result provides a basis for addressing learning problems by controlling such PDEs in future work.
\end{abstract}


\maketitle
\tableofcontents

% !TEX root =  ../continuous-domain_assignment_flows.tex
%%%%%%%%%%%%%%%%%%%%%%%%%%%%%%%%%%%%%%%%%%

\section{Introduction}

Deep networks are omnipresent in many disciplines due to their unprecedented predictive power and the availability of software for training that is easy to use. However, this rapid development during recent years has not improved our mathematical understanding in the same way, so far \cite{Elad:2017aa}. The `black box' behaviour of deep networks and systematic failures \cite{Antun:2019aa}, the lack of performance guarantees and reproducibility of results, raises doubts if a purely data-driven approach can deliver the high expectations of some of its most passionate proponents \cite{Coveney:2016aa}. `Mathematics of deep networks', therefore, has become a focal point of research.

Initiated maybe by \cite{He:2016aa} and mathematically substantiated and promoted by, e.g.~\cite{Haber:2017aa,E:2017aa}, attempts to understand deep network architectures as discretized realizations of dynamical systems has become a fruitful line of research. Adopting this viewpoint, we introduced a dynamical system -- called \textit{assignent flow} -- for contextual data classification and image labeling on graphs \cite{Astrom:2017ac}. We refer to \cite{Schnorr:2019aa} for a review of recent work including parameter estimation (learning) \cite{Huhnerbein:2019aa}, adaption of data prototypes during assignment \cite{Zern:2019aa}, and learning prototypes from low-rank data representations and self-assignment \cite{Zisler:2019aa}.

Two key properties of the assignment flow are \textit{smoothness} and \textit{gradual enforcement} of unambigious classification in a \textit{single} process, solely induced by adopting an elementary statistical manifold as state space that is natural for classification tasks, and the corresponding information geometry \cite{Amari:2000aa}. This differs from traditional \textit{variational} approaches to image labeling \cite{Lellmann:2011aa,Chambolle:2012aa} that enjoy convexity but are inherently nonsmooth and require postprocessing to achieve unambigous decisions. We regard nonsmoothness as a major barrier to the design of hierarchical architectures for data classification.

The assignment flow combines by \textit{composition} (rather than by addition) separate local processes at each vertex of the underlying graph and nonlocal regularization. Each local process for label assignment is governed by an ODE, the \textit{replicator equation} \cite{Hofbauer:2003aa,Sandholm:2010aa}, whereas regularization is accomplished by nonlocal \textit{geometric} averaging of the evolving assignments. It is well known \cite{Hofbauer:2003aa} that if the affinity measure which defines the replicator equation and hence governs label selection can be derived as gradient of a potential, then the replicator equation is just the corresponding Riemannian gradient flow induced by the Fisher-Rao metric. The geometric regularization of assignments performed by the assignment flow yields an affinity measure for which the (non-)existence of a corresponding potential is not immediate, however.

\vspace{0.25cm}
\textbf{Contribution and Organization.} The objective of this paper is to clarify this situation. After collecting background material in Section \ref{sec:Preliminaries}, we prove that no potential exists that enables to characterize the assignment flow as Riemannian gradient flow (Section \ref{sec:NovelRepAF:non-potential-flow}). Next, we provide a novel parametrization of the assignment flow by separating a dominant component of the flow, called \textit{$S$-flow}, that completely determines the remaining component and hence essentially characterizes the assignment flow (Section \ref{sec:NovelRepAF:S-parametrization}). The $S$-flow \textit{does} correspond to a potential, under an additional symmetry assumption with respect to the weights that parametrize the regularization properties of the assignment flow through (weighted) geometric averaging. This potential can be decomposed into two components that make explicit the two interacting processes mentioned above: regularization of label assignments and gradually enforcing unambigous decisions. We point out again that this is a direct consequence of the `spherical geometry' (positive curvature) underlying the assignment flow.

Based on this result, we consider the corresponding \textit{continuous-domain variational} formulation in Section \ref{sec:zero-scale-functional}. We prove well-posedness which is not immediate due to nonconvexity, and we propose an algorithm that computes a locally optimal assignment by solving a sequence of simple \textit{linear} PDEs, with changing right-hand side and subject to a simple convex constraint. A numerical example demonstrates that our PDE-based approach reproduces results obtained with solving the original formulation of the assignment flow using completely different numerical techniques \cite{Zeilmann:2018aa}. We hope that the \textit{simplicity} of our PDE-approach and the direct connection to a smooth geometric setting will stimulate future work on learning, from an optimal control point-of-view \cite{E:2019aa,Liu:2019ab}. We conclude by a formal derivation of a PDE that characterizes global minimizers of the nonconvex objective function (Section \ref{sec:VI-PDE}) and by outlining future research in Section \ref{sec:Conclusion}. 

% !TEX root =  ../continuous-domain_assignment_flows.tex
%%%%%%%%%%%%%%%%%%%%%%%%%%%%%%%%%%%%%%%%%%

\section{Preliminaries}\label{sec:Preliminaries}
\subsection{Basic Notation}

We denote the standard basis of $\R^n$ by
\begin{equation}\label{eq:def-e1-en}
\mc{B}_{n} := \{e_1, \ldots, e_n\}.
\end{equation}
$|\cdot|$ applied to a finite set denotes its cardinality, i.e.~$|\mc{B}_{n}|=n$.
We set $[n]=\{1,2,\dotsc,n\}$ for $n \in \N$ and $\eins_{n} = (1,1,\dotsc,1)^{\T} \in \R^{n}$. The symbols
\begin{equation}
\mc{I}=[n],\qquad \mc{J}=[c],\qquad n,c\in \N
\end{equation}
will specifically index data points and classes (labels), respectively. $\|\cdot\|$ denotes the Euclidean vector norm and the Frobenius matrix norm induced by the inner product $\|A\|=\la A,A \ra^{1/2} = \tr(A^{\T} A)^{1/2}$. All other norms will be indicated by a corresponding subscript. For a given matrix $A \in \R^{n \times c}$, $A_{i},\,i \in [n]$ denote the row vectors, $A^{j},\,j \in [c]$ denote the column vectors, and $A^{\T} \in \R^{c \times n}$ the transpose matrix. $\mb{S}_{+}^{n}$ denotes the set of all symmetric $n \times n$ matrices with nonnegative entries.
\begin{equation}\label{eq:def-Delta-n}
\Delta_{n} = \{p \in \R_{+}^{n} \colon \la\eins_{n},p\ra=1\}
\end{equation}
denotes the probability simplex. There will be no danger to confuse it with the Laplacian differential operator $\Delta$ that we use without subscript. For strictly positive vectors $p > 0$, we efficiently denote componentwise subdivision by $\frac{v}{p}$. Likewise, we set $p v = (p_{1} v_{1},\dotsc, p_{n} v_{n})^{\T}$. The exponential \textit{function} applies componentwise to vectors (and similarly for $\log$) and will always be denoted by $e^{v}=(e^{v_{1}},\dotsc,e^{v_{n}})^{\T}$, in order not to confuse it with the exponential \textit{maps} \eqref{eq:schnoerr-eq:exp-maps}.

Strong and weak convergence of a sequence $(f_{n})$ is written as $f_{n} \to f$ and $f_{n} \rightharpoonup f$, respectively. $\psi_{S}$ denotes the indicator function of some set $S$: $\psi_{S}(i) = 1$ if $i \in S$ and $\psi_{S}(i) = 0$ otherwise. $\delta_{C}$ denotes the indicator function from the optimization point of view: $\delta_{C}(f) = 0$ if $f \in C$ and $\delta_{C}(f) = +\infty$ otherwise.

\subsection{Assignment Manifold and Flow}
\label{sec:Preliminaries-AF}
We sketch the assignment flow as introduced by \cite{Astrom:2017ac} and refer to the recent survey \cite{Schnorr:2019aa} for more background and a review of recent related work.

\subsubsection{Assignment Manifold}\label{sec:assignmentManifold}
Let $(\mc{F},d_{\mc{F}})$ be a metric space and
\begin{equation}\label{eq:def-mcF-n}
\mc{F}_{n} = \{f_{i} \in \mc{F} \colon i \in \mc{I}\},\qquad |\mc{I}|=n.
\end{equation}
given data. Assume that a predefined set of prototypes
\begin{equation}\label{eq:def-mcF-ast}
\mc{F}_{\ast} = \{f^{\ast}_{j} \in \mc{F} \colon j \in \mc{J}\},\qquad |\mc{J}|=c.
\end{equation}
is given. \textit{Data labeling} denotes the assignments
\begin{equation}\label{eq:fj-fi}
j \to i,\qquad f_{j}^{\ast} \to f_{i}
\end{equation}
of a single prototype $f_{j}^{\ast} \in \mc{F}_{\ast}$ to each data point $f_{i} \in \mc{F}_{n}$.
The set $\mc{I}$ is assumed to form the vertex set of an  undirected graph $\mc{G}=(\mc{I},\mc{E})$ which defines a relation $\mc{E} \subset \mc{I} \times \mc{I}$ and neighborhoods
\begin{equation}\label{eq:def-Ni}
\mc{N}_{i} = \{k \in \mc{I} \colon ik \in \mc{E}\} \cup \{i\},
\end{equation}
where $ik$ is a shorthand for the unordered pair (edge) $(i,k)=(k,i)$. We require these neighborhoods to satisfy the relations
\begin{equation}\label{eq:Ni-Nk}
k \in \mc{N}_{i} \quad\gdw\quad i \in \mc{N}_{k},\qquad\forall i,k \in \mc{I}.
\end{equation}

The assignments (labeling) \eqref{eq:fj-fi} are represented by matrices in the set
\begin{equation}\label{eq:def-W-ast}
\mc{W}_{\ast} = \{W \in \{0,1\}^{n \times c} \colon W\eins_{c}=\eins_{n}\}
\end{equation}
with unit vectors $W_{i},\,i \in \mc{I}$, called \textit{assignment vectors}, as row vectors. These assignment vectors are computed by numerically integrating the assignment flow below \eqref{eq:assignment-flow}, in the following elementary geometric setting. The integrality constraint of \eqref{eq:def-W-ast} is relaxed and vectors
\begin{equation}\label{eq:def-Wi}
W_{i} = (W_{i1},\dotsc,W_{ic})^{\T} \in \mc{S},\quad i \in \mc{I},
\end{equation}
that we still call \textit{assignment vectors},
are considered
on the elementary Riemannian manifold
\begin{equation}\label{eq:def-S}
(\mc{S},g),\qquad
\mc{S} = \{p \in \Delta_{c} \colon p > 0\}
\end{equation}
with
\begin{equation}\label{eq:barycenter-S}
\eins_{\mc{S}} = \frac{1}{c}\eins_{c} \in \mc{S},
\qquad(\textit{barycenter})
\end{equation}
tangent space
\begin{equation}\label{eq:def-T0}
T_{0}
= \{v \in \R^{c} \colon \la\eins_{c},v \ra=0\}
\end{equation}
and tangent bundle $T\mc{S} = \mc{S} \times T_{0}$,
orthogonal projection
\begin{equation}\label{eq:def-Pi0}
\Pi_{0} \colon \R^{c} \to T_{0},\qquad
\Pi_{0} = I - \eins_{\mc{S}}\eins^{\T}
\end{equation}
and the Fisher-Rao metric
\begin{equation}\label{schnoerr-eq:FR-metric-S}
g_{p}(u,v) = \sum_{j \in \mc{J}} \frac{u^{j} {v}^{j}}{p^{j}},\quad p \in \mc{S},\quad
u,v \in T_{0}.
\end{equation}
Based on the linear map
\begin{equation}\label{eq:def-Rp}
R_{p} \colon \R^{c} \to T_{0},\qquad
R_{p} = \Diag(p)-p p^{\T},\qquad p \in \mc{S}
\end{equation}
satisfying
\begin{equation}\label{eq:Rp-Pi0}
R_{p} = R_{p} \Pi_{0} = \Pi_{0} R_{p},
\end{equation}
exponential maps and their inverses are defined as
\begin{subequations}\label{eq:schnoerr-eq:exp-maps}
\begin{align}
\Exp &\colon \mc{S} \times T_{0} \to \mc{S}, &
(p,v) &\mapsto
\Exp_{p}(v) = \frac{p e^{\frac{v}{p}}}{\la p,e^{\frac{v}{p}}\ra},
\label{eq:Exp0} \\ \label{eq:IExp0}
\Exp_{p}^{-1} &\colon \mc{S} \to T_{0}, &
q &\mapsto  \Exp_{p}^{-1}(q) = R_{p} \log\frac{q}{p},
\\
\exp_{p} &\colon T_{0} \to \mc{S}, &
\exp_{p} &= \Exp_{p} \circ R_{p},
\label{schnoerr-eq:def-exp-p} \\ \label{schnoerr-eq:def-exp-p-inverse}
\exp_{p}^{-1} &\colon \mc{S} \to T_{0}, &
\exp_{p}^{-1}(q) &= \Pi_{0}\log\frac{q}{p}.
\end{align}
\end{subequations}
Applying the map $\exp_{p}$ to a vector in $\R^{c} = T_{0} \oplus \R\eins$ does not depend on the constant component of the argument, due to \eqref{eq:Rp-Pi0}.

\begin{remark}\label{rem:Exp-map}
The map $\Exp$ corresponds to the e-connection of information geometry \cite{Amari:2000aa}, rather than to the exponential map of the Riemannian connection. Accordingly, the affine geodesics \eqref{eq:Exp0} are not length-minimizing. But they provide an close approximation \cite[Prop.~3]{Astrom:2017ac} and are more convenient for numerical computations.
\end{remark}

The \textit{assignment manifold} is defined as
\begin{equation}\label{schnoerr-eq:def-mcW}
(\mc{W},g),\qquad \mc{W} = \mc{S} \times\dotsb\times \mc{S}.\qquad (n = |\mc{I}|\;\text{factors})
\end{equation}
We identify $\mc{W}$ with the embedding into $\R^{n\times c}$
\begin{equation}\label{eq:mcW-matrix-embed}
  \mc{W} = \{ W \in \R^{n\times c} \colon W\eins_c = \eins_n\text{ and } W_{ij} > 0 \text{ for all } i\in[n], j\in [c]\}.
\end{equation}
Thus, points $W \in \mc{W}$ are row-stochastic matrices $W \in \R^{n \times c}$ with row vectors $W_{i} \in \mc{S},\; i \in \mc{I}$ that represent the assignments \eqref{eq:fj-fi} for every $i \in \mc{I}$. We set
\begin{equation}\label{schnoerr-eq:TmcW}
\mc{T}_{0} := T_{0} \times\dotsb\times T_{0}
\qquad (n = |\mc{I}|\;\text{factors}).
\end{equation}
Due to \eqref{eq:mcW-matrix-embed}, the tangent space $\mc{T}_0$ can be identified with
\begin{equation}\label{eq:mcT-matrix-embed}
  \mc{T}_0 = \{ V \in \R^{n\times c} \colon V\eins_c = 0\}.
\end{equation}
Thus, $V_i \in T_{0}$ for all row vectors of $V \in \R^{n \times c}$ and $i \in \mc{I}$. All mappings defined above factorize in a natural way and apply row-wise, e.g.~$\Exp_{W} = (\Exp_{W_{1}},\dotsc,\Exp_{W_{n}})$ etc.

\subsubsection{Assignment Flow}

Based on \eqref{eq:def-mcF-n} and \eqref{eq:def-mcF-ast}, the distance vector field
\begin{equation}\label{eq:def-distance-vector}
D_{\mc{F};i} = \big(d_{\mc{F}}(f_{i},f_{1}^{\ast}),\dotsc,d_{\mc{F}}(f_{i},f_{c}^{\ast})\big)^{\T},\qquad i \in \mc{I}
\end{equation}
is well-defined. These vectors are collected as row vectors of the \textit{distance matrix}
\begin{equation}\label{eq:def-distance-matrix}
D_{\mc{F}} \in \mb{S}_{+}^{n}.
\end{equation}
The \textit{likelihood map} and the \textit{likelihood vectors}, respectively, are defined as
\begin{equation}\label{schnoerr-eq:def-Li}
L_{i} \colon \mc{S} \to \mc{S},\qquad
L_{i}(W_{i})
= \exp_{W_{i}}\Big(-\frac{1}{\rho}D_{\mc{F};i}\Big)
= \frac{W_{i} e^{-\frac{1}{\rho} D_{\mc{F};i}}}{\la W_{i},e^{-\frac{1}{\rho} D_{\mc{F};i}} \ra},\qquad i \in \mc{I},
\end{equation}
where the scaling parameter $\rho > 0$ is used for normalizing the a-prior unknown scale of the components of $D_{\mc{F};i}$ that depends on the specific application at hand.

A key component of the assignment flow is the interaction of the likelihood vectors through \textit{geometric} averaging within the local neighborhoods \eqref{eq:def-Ni}. Specifically, using  weights
\begin{equation}\label{eq:weights-Omega-i}
% \Omega_{i} = \Big\{w_{i,k} \colon k \in \mc{N}_{i},\; w_{i,k} > 0,\; \sum_{k \in \mc{N}_{i}} w_{i,k}=1\Big\},\quad i \in \mc{I},
  \omega_{ik} > 0\quad \text{for all}\; k \in \mc{N}_{i},\;i \in \mc{I}\quad\text{with}\quad \sum_{k \in \mc{N}_{i}} w_{ik}=1,
\end{equation}
the \textit{similarity map} and the \textit{similarity vectors}, respectively, are defined as
\begin{equation}\label{eq:def-Si}
S_{i} \colon \mc{W} \to \mc{S},\qquad
S_{i}(W) = \Exp_{W_{i}}\Big(\sum_{k \in \mc{N}_{i}} w_{ik} \Exp_{W_{i}}^{-1}\big(L_{k}(W_{k})\big)\Big),\qquad i \in \mc{I}.
\end{equation}
If $\Exp_{W_{i}}$ were the exponential map of the Riemannian (Levi-Civita) connection, then the argument inside the brackets of the right-hand side would just be the negative Riemannian gradient with respect to $W_{i}$ of the center of mass objective function comprising the points $L_{k},\,k \in \mc{N}_{i}$, i.e.~the weighted sum of the squared Riemannian distances between $W_{i}$ and  $L_{k}$   \cite[Lemma 6.9.4]{Jost:2017aa}. In view of Remark \ref{rem:Exp-map}, this interpretation is only approximately true mathematically, but still correct informally: $S_{i}(W)$ moves $W_{i}$ towards the geometric mean of the likelihood vectors $L_{k},\,k \in \mc{N}_{i}$. Since $\Exp_{W_{i}}(0)=W_{i}$, this mean is equal to $W_{i}$ if the aforementioned gradient vanishes.

The \textit{assignment flow} is induced by the locally coupled system of nonlinear ODEs
\begin{subequations}\label{eq:assignment-flow}
\begin{align}
\dot W &= R_{W}S(W),\qquad W(0)=\eins_{\mc{W}},
\label{eq:assignment-flow-a}
\\
\label{eq:assignment-flow-b}
\dot W_{i} &= R_{W_{i}} S_{i}(W),\qquad W_{i}(0)=\eins_{\mc{S}},\quad i \in \mc{I},
\end{align}
\end{subequations}
where $\eins_{\mc{W}} \in \mc{W}$ denotes the barycenter of the assignment manifold \eqref{schnoerr-eq:def-mcW}. The solution curve $W(t)\in\mc{W}$ is numerically computed by geometric integration \cite{Zeilmann:2018aa} and determines a labeling $W(T) \in \mc{W}_{\ast}$ for sufficiently large $T$ after a trivial rounding operation.

%%%%%%%%%%%%%%%%%%%%%%
\subsection{Functional Analysis}
We record background material that will be used in Section \ref{sec:zero-scale-functional}.
\subsubsection{Sobolev Spaces}
\label{sec:Preliminaries-Sobolev}

We list few basic definitions and fix the corresponding notation \cite{Ziemer:1989aa,Attouch:2014aa}. Throughout this section $\Omega \subset \R^{d}$ denotes an open bounded domain.

We denote the inner product and the norm of functions $f,g \in L^{2}(\Omega)$ by
\begin{equation}
(f,g)_{\Omega} = \int_{\Omega} f g dx,\qquad
\|f\|_{\Omega} = (f,f)_{\Omega}^{1/2}.
\end{equation}
Functions $f_{1}$ and $f_{2}$ are equivalent and identified whenever they merely differ pointwise on a Lebesque-negligible set of measure zero. $f_{1}$ and $f_{2}$ then are said to be equal a.e. (almost everywhere). $H^{1}(\Omega)=W^{1,2}(\Omega)$ denotes the Sobolev space of functions $f$ with square-integrable weak derivatives $D^{\alpha} f$ up to order one. $H^{1}(\Omega)$ is a Hilbert space with inner product and norm denoted by
\begin{equation}\label{eq:f-1-norm-scalar} (f,g)_{1;\Omega} = \sum_{|\alpha|\leq 1} (D^{\alpha} f,D^{\alpha} g)_{\Omega},\qquad
\|f\|_{1;\Omega} = \Big(\sum_{|\alpha|\leq 1}\|D^{\alpha} f\|_{\Omega}^{2}\Big)^{1/2}.
\end{equation}
\begin{lemma}[{\cite[Cor.~2.1.9]{Ziemer:1989aa}}]\label{lem:Du-0-constant}
If $\Omega$ is connected, $u \in H^{1}(\Omega)$ and $Du=0$ a.e.~on $\Omega$, then $u$ is equivalent to a constant function on $\Omega$.
\end{lemma}
The closure in $H^{1}(\Omega)$ of the set of test functions $C^{\infty}_{c}(\Omega)$ that are compactly supported on $\Omega$, is the Sobolev space
\begin{equation}
H_{0}^{1}(\Omega) = \ol{C^{\infty}_{c}(\Omega)} \subset H^{1}(\Omega).
\end{equation}
It contains all functions in $H^{1}(\Omega)$ whose boundary values on $\partial\Omega$ (in the sense of traces) vanish. The space $H^{1}(\Omega;\R^{c}),\, 2 \leq c \in \N$ contains vector-valued functions $f$ whose component functions $f_{i},\, i \in [c]$ are in $H^{1}(\Omega)$. For notational efficiency, we denote the norm of $f \in H^{1}(\Omega;\R^{c})$ by
\begin{equation}
\|f\|_{1;\Omega} = \Big(\sum_{i \in [c]}\|f_{i}\|_{1;\Omega}^{2}\Big)^{1/2}
\end{equation}
as in the scalar case \eqref{eq:f-1-norm-scalar}.
It will be clear from the context if $f$ is scalar- or vector-valued.

The compactness theorem of Rellich-Kondrakov \cite[Thm.~5.3.3]{Attouch:2014aa} says that the canonical embedding
\begin{equation}
H_{0}^{1}(\Omega) \hookrightarrow L^{2}(\Omega)
\end{equation}
is compact, i.e.~every bounded subset in $H_{0}^{1}(\Omega)$ is relatively compact in $L^{2}(\Omega)$. This extends to the vector-valued case
\begin{equation}\label{eq:compact-embedding}
H_{0}^{1}(\Omega;\R^{c}) \hookrightarrow L^{2}(\Omega;\R^{c})
\end{equation}
since $H_{0}^{1}(\Omega;\R^{c})$ is isomorphic to $H_{0}^{1}(\Omega) \times \dotsb \times H_{0}^{1}(\Omega)$ and likewise for $L^{2}(\Omega;\R^{c})$. The dual space of $H_{0}^{1}(\Omega)$ is commonly denoted by $H^{-1}(\Omega)=\big(H_{0}^{1}(\Omega)\big)$. Accordingly, we set $H^{-1}(\Omega;\R^{c})=\big(H_{0}^{1}(\Omega;\R^{c})\big)'$.

\subsubsection{Weak Convergence Properties, Variational Inequalities}
\label{sec:Preliminaries-Existence-Minimizers}

We list few further basic facts \cite[Prop.~38.2]{Zeidler:1985aa}, \cite[Prop.~2.4.6]{Attouch:2014aa}.
\begin{proposition}\label{prop:properties-Banach}
The following assertions hold in a Banach space $X$.
\begin{enumerate}[(a)]
\item A closed convex subset $C \subset X$ is weakly closed, i.e.~a sequence $(f_{n})_{n \in \N} \subset C$ that weakly converges to $f$ implies $f \in C$.
\item If $X$ is reflexive (in particular, if $X$ is a Hilbert space), then every bounded sequence in $X$ has a weakly convergent subsequence.
\item If $f_{n}$ weakly converges to $f$, then $(f_{n})_{n \in \N}$ is bounded and
\begin{equation}
\|f\|_{X} \leq \liminf_{n\to\infty}\|f_{n}\|_{X}.
\end{equation}
\end{enumerate}
\end{proposition}
The following theorem states conditions for minimizers of the functional to satisfy a corresponding variational inequality.
\begin{theorem}[{\cite[Thm.~46.A(a)]{Zeidler:1985aa}}]\label{thm:VI-necessarily}
Let $F \colon C \to \R$ be a functional on the convex nonempty set $C$ of a real locally convex space $X$, and let $b \in X'$ be a given element. Suppose the Gateaux-derivative $F'$ exists on $C$. Then any solution $f$ of
\begin{equation}
\min_{f \in C}\big\{F(f)-\la b, f \ra_{X'\times X}\big\},
\end{equation}
satisfies the variational inequality
\begin{equation}
\la F'(f)-b,h-f \ra_{X'\times X} \geq 0,\qquad\text{for all}\; h \in C.
\end{equation}
\end{theorem}
%

% !TEX root =  ../continuous-domain_assignment_flows.tex
%%%%%%%%%%%%%%%%%%%%%%%%%%%%%%%%%%%%%%%%%%

\section{A Novel Representation of the Assignment Flow}

Let $J \colon \mc{W} \to \R$ be a smooth function on the assignment manifold \eqref{schnoerr-eq:def-mcW} and denote the Riemannian gradient of $J$ at $W \in\mc{W}$ induced by the Fisher-Rao metric \eqref{schnoerr-eq:FR-metric-S} by $\Rgrad J(W) \in \mc{T}_0$. In view of the embedding \eqref{eq:mcW-matrix-embed}, we can also compute the Euclidean gradient of $J$ denoted by $\Egrad J(W) \in \R^{n\times c}$. These two gradients are related by
\cite[Prop.~1]{Astrom:2017ac}
\begin{equation}\label{eq:NovelRepAF:relationRgradEgrad}
  \Rgrad J(W) = \RO_W \Egrad J(W),\qquad W \in \mc{W},
\end{equation}
where $\RO_{W} \colon \R^{n\times c} \to \mc{T}_0$ is the product map obtained by applying $\RO_{W_i}$ from \eqref{eq:def-Rp} to every row vector indexed by $i \in \mc{I}$. This relation raises the natural question: Is there a potential $J$ such that the assignment flow \eqref{eq:assignment-flow} is a Riemannian gradient descent flow with respect to $J$, i.e.~does %the equation
$\RO_W S(W) = -\Rgrad J(W)$ hold?

We next show that such a potential does not exist in general (Section \ref{sec:NovelRepAF:non-potential-flow}). However, in Section~\ref{sec:NovelRepAF:S-parametrization}, we derive a novel representation by decoupling the assignment flow into two separate flows, where one flow steers the other and in this sense dominates the assignment flow. Under the additional assumption that the weights $\omega_{ij}$ of the similarity map $S(W)$ in \eqref{eq:def-Si} are symmetric, we show that the dominating flow is a Riemannian gradient flow induced by a potential. This result is the basis for the continuous-domain formulation of the assignment flow studied in the subsequent sections.

\subsection{Non-Potential Flow}\label{sec:NovelRepAF:non-potential-flow}
We next show (Theorem~\ref{thm:nonexistance_of_Potential_for_assignment_flow}) that under some mild assumptions on $D_{\mc{F}}$ \eqref{eq:def-distance-matrix} which are always fulfilled in practice, no potential $J$ exists that induces the assignment flow. In order to prove this result, we first derive some properties of the mapping $\exp$ given by \eqref{schnoerr-eq:def-exp-p} as well as explicit expressions of the differential $dS(W)$ of the similarity map \eqref{eq:def-Si} and its transpose $dS(W)^\T$ with respect to the standard Euclidean structure on $\R^{n\times c}$.

\begin{lemma}\label{lem:prop_exp}
  The following properties hold for $\exp_p$ and its inverse \eqref{schnoerr-eq:def-exp-p}, \eqref{schnoerr-eq:def-exp-p-inverse}.
  \begin{enumerate}
    \item\label{enum:lem:prop_exp:expression_exp_dexp_and_dexpInv}
    For every $p \in \mc{S}$ the map $\exp_p \colon \R^c \to \mc{S}$ can be expressed by
    \begin{equation}\label{eq:lem:prop_exp:exp_explicitFormula}
      v \mapsto \exp_p(v) = \frac{pe^v}{\la p, e^v\ra}.
    \end{equation}
    Its restriction to $T_0$, $\exp_p \colon T_0 \to \mc{S}$, is a diffeomorphism. The differential of $\exp_p$ and $\exp_p^{-1}$ at $v\in T_0$ and $q \in \mc{S}$, respectively, are given by
    \begin{equation}\label{eq:lem:prop_exp:dexp_and_dexpInv}
      d\exp_p(v)[u] = \RO_{\exp_p(v)}[u]\quad \text{and}\quad d\exp_p^{-1}(q)[u] = \PT\Big[\frac{u}{q}\Big] \quad \text{for all } u \in T_0.
    \end{equation}

    \item\label{enum:lem:prop_exp:ExpP_expressed_with_exp}
    Let $p, q \in \mc{S}$. Then $\Exp_p^{-1}(q) = \RO_p \exp_p^{-1}(q)$.

    \item \label{enum:lem:prop_exp:RqIso}
    Let $q \in \mc{S}$. If the linear map $R_q$ from \eqref{eq:def-Rp} is restricted to $T_0$, then $R_q \colon T_0 \to T_0$ is a linear isomorphism with inverse given by $(R_q|_{T_0})^{-1}(u) = \PT \Big[ \frac{u}{q}\Big]$ for all $u\in T_0$.

    \item\label{enum:lem:prop_exp:exp_GroupAction}
    If $\R^c$ is viewed as an abelian group, then $\exp \colon \R^c \times \mc{S} \to \mc{S}$ given by $(v, p) \mapsto \exp_p(v)$ defines a Lie-group action, i.e.
    \begin{equation}\label{eq:lem:prop_exp:groupActionProperties}
      \exp_{p}(v + u) = \exp_{\exp_p(u)}(v)\quad\text{and}\quad \exp_{p}(0) = p \quad \quad \text{for all } v, u \in T_0 \text{ and } p \in \mc{S}.
    \end{equation}
    Furthermore, the following identities follow for all $p, q, a \in \mc{S}$ and $v \in \R^c$
    \begin{subequations}
    \begin{align}
      \exp_{p}(v) &= \exp_{q}\big( v + \exp_{q}^{-1}(p)\big)\label{eq:expP_expressed_with_expQ}\\
      \exp_q^{-1}(p) &= -\exp_p^{-1}(q)\label{eq:expPInv_signFlip}\\
      \exp_q^{-1}(a) &= \exp^{-1}_p(a) - \exp_p^{-1}(q).\label{eq:expQInv_expressed_with_expPInv}
    \end{align}
    \end{subequations}
  \end{enumerate}
\end{lemma}
\begin{proof}
 \eqref{enum:lem:prop_exp:expression_exp_dexp_and_dexpInv}: We have $\Exp_p(v + \lambda p) = \Exp_p(v)$ for every $p\in \mc{S}$, $v \in T_0$ and $\lambda \in \R$, as a simple computation using definition \eqref{eq:Exp0} of $\Exp_p$ directly shows. Therefore, for every $v \in T_0$
  \begin{equation}
    \exp_p(v) = \Exp_p(\RO_p v) = \Exp_p(pv - \la v, p\ra p) = \Exp_p(pv) = \frac{pe^v}{\la p, e^v\ra}.
  \end{equation}
  If we restrict $\exp_p$ to $T_0$, then an inverse is explicitly given by \eqref{schnoerr-eq:def-exp-p}. The differentials \eqref{eq:lem:prop_exp:dexp_and_dexpInv} result from a standard computation.

  \eqref{enum:lem:prop_exp:ExpP_expressed_with_exp}: The formula is a direct consequence of the formulas for $\Exp_p^{-1}$ and $\exp_p^{-1}$ given in \eqref{eq:IExp0} and \eqref{schnoerr-eq:def-exp-p}, together with the fact \eqref{eq:Rp-Pi0}.% that $R_p \PT = R_p$.

  \eqref{enum:lem:prop_exp:RqIso}: Fix any $p \in \mc{S}$ and set $v_q := \exp_p^{-1}(q)$ for $q \in \mc{S}$. Since $\exp_p \colon T_0 \to \mc{S}$ is a diffeomorphism, the differential $d\exp_p(v_q) \colon T_0 \to T_0$ is an isomorphism. By \eqref{eq:lem:prop_exp:dexp_and_dexpInv}, we have $\RO_q[u] = \RO_{\exp_p(v_q)}[u] = d\exp_p(v_q)[u]$ for all $u \in T_0$, showing that $\RO_q$ is an isomorphism with the corresponding inverse.

  \eqref{enum:lem:prop_exp:exp_GroupAction}: Properties \eqref{eq:lem:prop_exp:groupActionProperties} defining the group action are directly verified using \eqref{eq:lem:prop_exp:exp_explicitFormula}. Now, suppose $p, q, a\in \mc{S}$ and $v \in \R^c$ are arbitrary. Since $\exp_q \colon T_0 \to \mc{S}$ is a diffeomorphism, we have $p = \exp_q\big( \exp_q^{-1}(p)\big)$ and by the group action property
  \begin{equation}
    \exp_{p}(v) = \exp_{\exp_{q}\big( \exp_{q}^{-1}(p)\big)}(v) = \exp_{q}\big( v + \exp_{q}^{-1}(p)\big),
  \end{equation}
  which proves \eqref{eq:expP_expressed_with_expQ}. To show \eqref{eq:expPInv_signFlip}, set $v_a := \exp_p^{-1}(a)$ and substitute this vector into \eqref{eq:expP_expressed_with_expQ}. Applying $\exp_q^{-1}$ to both sides then gives
  \begin{equation}\label{eq:expPInv_expressed_with_expQInv_EQ1}
    \exp_q^{-1}(a) = \exp_{q}^{-1}\big( \exp_p(v_a) \big) = v_a + \exp_q^{-1}(p) = \exp^{-1}_p(a) + \exp_q^{-1}(p).
  \end{equation}
  Setting $a = q$ in this equation, we obtain \eqref{eq:expPInv_signFlip} from
  \begin{equation}
    0 = \exp_q^{-1}(q) = \exp_p^{-1}(q) + \exp_q^{-1}(p).
  \end{equation}
  Using $\exp_q^{-1}(p) = -\exp_p^{-1}(q)$ in \eqref{eq:expPInv_expressed_with_expQInv_EQ1} yields \eqref{eq:expQInv_expressed_with_expPInv}.
\end{proof}

\begin{lemma}\label{lem:Si_expressed_by_expBS}
  The $i$-th component of the similarity map $S(W)$ defined by \eqref{eq:def-Si} can equivalently be expressed as
  \begin{equation}
    S_i(W) = \exp_{\BS}\Big( \sum_{j\in \mc{N}_i} \omega_{ij}\Big(\exp_{\BS}^{-1}(W_j) - \frac{1}{\rho} D_{\mc{F};j}\Big)\Big)\quad
    \text{for all } i\in \mc{I} \text{ and } W \in \mc{W}.
  \end{equation}
\end{lemma}
\begin{proof}
  Consider the expression $\Exp^{-1}_{W_i}\big(L_j(W_j)\big)$ in the sum of the definition \eqref{eq:def-Si} of $S_i(W)$. Using \eqref{eq:lem:prop_exp:exp_explicitFormula} and \eqref{eq:expP_expressed_with_expQ}, the likelihood \eqref{schnoerr-eq:def-Li} can be expressed as
  \begin{equation}
    L_j(W_j) = \exp_{W_j}\Big(-\frac{1}{\rho}D_{\mc{F};j}\Big) = \exp_{W_i}\Big( \exp_{W_i}^{-1}(W_j) - \frac{1}{\rho} D_{\mc{F};j}\Big).
  \end{equation}
  In the following, we set
  \begin{equation}\label{eq:proof-def-Vk}
  V_k = \exp_{\BS}^{-1}(W_k) \qquad
  \text{for all}\;k \in \mc{I}.
  \end{equation}
  With this and \eqref{eq:expQInv_expressed_with_expPInv}, we have
  \begin{equation}
    \exp_{W_i}^{-1}(W_j) = \exp_{\BS}^{-1}(W_j) - \exp_{\BS}^{-1}(W_i) = V_j - V_i.
  \end{equation}
  The two previous identities and Lemma~\ref{lem:prop_exp}\eqref{enum:lem:prop_exp:ExpP_expressed_with_exp} give
  \begin{subequations}
  \begin{align}
    \Exp^{-1}_{W_i}\big(L_j(W_j)\big) &= \RO_{W_i}\Big[ \exp_{W_i}^{-1}\big(L_j(W_j)\big)\Big] = \RO_{W_i}\Big[\exp_{W_i}^{-1}(W_j) - \frac{1}{\rho}D_{\mc{F};j}\Big]\\
    &= \RO_{W_i}\Big[V_j - V_i - \frac{1}{\rho}D_{\mc{F};j}\Big].
  \end{align}
  \end{subequations}
  The sum over the neighboring nodes $\mc{N}_i$ in the definition \eqref{eq:def-Si} of $S_i(W)$ can therefore be rewritten as
  \begin{subequations}
  \begin{align}
    \sum_{j \in \mc{N}_i} \omega_{ij}\Exp^{-1}_{W_i}\big(L_j(W_j)\big) &= \sum_{j \in \mc{N}_i} \omega_{ij}\RO_{W_i}\Big[V_j - V_i - \frac{1}{\rho}D_{\mc{F};j}\Big]\\
    &= \RO_{W_i}\Big[- V_i + \sum_{j \in \mc{N}_i} \omega_{ij}\Big(V_j - \frac{1}{\rho}D_{\mc{F};j}\Big)\Big],
  \end{align}
  \end{subequations}
  where we used $\sum_{j\in\mc{N}_i} \omega_{ij} = 1$ for the last equation. Setting $Y_i := \sum_{j \in \mc{N}_i} \omega_{ij}\Big(V_j - \frac{1}{\rho}D_{\mc{F};j}\Big)$, we then have
  \begin{equation}
    S_i(W) = \Exp_{W_i}\big( \RO_{W_i}\big[- V_i + Y_i \big]\big) = \exp_{W_i}\big( -V_i + Y_i \big) = \exp_{\BS}\big( Y_i \big),
  \end{equation}
  where the last equality again follows from \eqref{eq:expP_expressed_with_expQ} together with the definition \eqref{eq:proof-def-Vk} of $V_i$.
\end{proof}

\begin{lemma}\label{lem:dS_and_dST}
  The $i$-th component of the differential of the similarity map $S(W) \in \mc{W}$ is given by
  \begin{equation}
    dS_i(W)[X] = \sum_{j \in \mc{N}_i} \omega_{ij} \RO_{S_i(W)}\left[\frac{X_j}{W_j}\right] \quad \text{for all } X \in \mc{T}_0\text{ and }  i \in \mc{I}.
  \end{equation}
  Furthermore, the $i$-th component of the adjoint differential $dS(W)^\T\colon \mc{T}_0 \to \mc{T}_0$ with respect to the standard Euclidean inner product on $\mc{T}_0 \subset \R^{n\times c}$ is given by
  \begin{equation}
    dS_i(W)^\T[X] = \sum_{j\in\mc{N}_i} \omega_{ji} \PT\left[ \frac{\RO_{S_j(W)}X_j}{W_i} \right] \quad \text{ for every } X \in \mc{T}_0\text{ and } i \in \mc{I}.
  \end{equation}
\end{lemma}
\begin{proof}
  Define the map $F_i \colon \mc{W} \to \R^c$ by $F_i(W) := \sum_{j\in\mc{N}_i} \omega_{ij}\Big(\exp_{\BS}^{-1}(W_j) - \frac{1}{\rho}D_{\mc{F};j}\Big) \in \R^c$ for all  $W \in \mc{W}$. Let $\gamma \colon (-\veps, \veps) \to \mc{W}$ be a smooth curve, with $\veps>0$, $\gamma(0) = W$ and $\dot{\gamma}(0) = X$. By \eqref{eq:lem:prop_exp:dexp_and_dexpInv}, we then have
  \begin{equation}
    dF_i(W)[X] = \frac{d}{dt} F_i(\gamma(t))\big|_{t = 0} = \sum_{j\in\mc{N}_i} \omega_{ij} \frac{d}{dt} \exp_{\BS}^{-1}(\gamma_j(t))\big|_{t = 0}
    = \sum_{j\in\mc{N}_i} \omega_{ij} \PT\Big[\frac{X_j}{W_j}\Big].
  \end{equation}
  Due to Lemma~\ref{lem:Si_expressed_by_expBS}, we can express the $i$-th component of the similarity map as $S_i(W) = \exp_{\BS}\big(F_i(W)\big)$. Therefore, the differential of $S_i$ is given by
  \begin{subequations}
  \begin{align}
    d S_i(W)[X] &= d\exp_{\BS}(F_i(W))\big[dF_i(W)[X]\big] = \RO_{\exp_{\BS}(F_i(W))}\big[dF_i(W)[X]\big]\\
    &= \RO_{S_i(W)}\Big[\sum_{j\in\mc{N}_i} \omega_{ij} \PT\Big[\frac{X_j}{W_j}\Big]\Big] = \sum_{j\in\mc{N}_i} \omega_{ij} \RO_{S_i(W)}\Big[\frac{X_j}{W_j}\Big],
  \end{align}
  \end{subequations}
  where we used $\RO_{S_i(W)}\PT = \RO_{S_i(W)}$ from \eqref{eq:Rp-Pi0} to obtain the last equation.

  Now let $W \in \mc{W}$ and $X, Y\in\mc{T}_0$ be arbitrary. By assumption on the neighborhood structure \eqref{eq:def-Ni}, we have $j \in \mc{N}_i$ if and only if $i \in \mc{N}_j$, i.e.~$\psi_{\mc{N}_i}(j) = \psi_{\mc{N}_j}(i)$. Since $\RO_{S_i(W)}\in \R^{c\times c}$ is a symmetric matrix, we obtain
  \begin{subequations}\label{eq:lem:dS_and_dST_EQ1}
  \begin{align}
    &\Big\la dS(W)[X], Y\Big\ra = \sum_{i \in\mc{I}} \Big\la dS_i(W)[X], Y_i\Big\ra = \sum_{i\in\mc{I}} \sum_{j\in\mc{N}_i} \omega_{ij}\Big\la \RO_{S_i(W)} \Big[\frac{X_j}{W_j}\Big], Y_i\Big\ra\\
    &= \sum_{i\in\mc{I}} \sum_{j\in\mc{I}}\psi_{\mc{N}_i}(j) \omega_{ij}\Big\la \frac{X_j}{W_j}, \RO_{S_i(W)}[Y_i]\Big\ra
    = \sum_{i\in\mc{I}} \sum_{j\in\mc{I}}\psi_{\mc{N}_j}(i) \omega_{ij}\Big\la X_j, \frac{\RO_{S_i(W)}[Y_i]}{W_j}\Big\ra\\
    &= \sum_{j\in\mc{I}} \sum_{i\in\mc{N}_j} \omega_{ij}\Big\la X_j, \PT\Big[\frac{\RO_{S_i(W)}[Y_i]}{W_j}\Big]\Big\ra = \sum_{j\in\mc{I}} \Big\la X_j, \sum_{i\in\mc{N}_j} \omega_{ij}\PT\Big[\frac{\RO_{S_i(W)}[Y_i]}{W_j}\Big]\Big\ra.
  \end{align}
  \end{subequations}
  On the other hand, we have
  \begin{equation}\label{eq:lem:dS_and_dST_EQ2}
    \Big\la dS(W)[X], Y\Big\ra = \Big\la X, dS(W)^\T[Y]\Big\ra = \sum_{j\in\mc{I}}\Big\la X_j, dS_j(W)^\T[Y]\Big\ra.
  \end{equation}
  Because \eqref{eq:lem:dS_and_dST_EQ1} and \ref{eq:lem:dS_and_dST_EQ2} hold for all $X, Y \in \mc{T}_0$, the formula for $dS_i(W)^\T[X]$ is proven.
\end{proof}

\begin{theorem}\label{thm:nonexistance_of_Potential_for_assignment_flow}
  Suppose $c \geq 3$ and there exists a node $i_0 \in \mc{I}$ such that the distance vector $D_{\mc{F};i_0}$ is not constant: $D_{\mc{F};i_0} \notin \R \eins$. Then no potential $J \colon \mc{W} \to \R$ exists satisfying $\RO_W S(W) = -\Rgrad J(W)$, i.e.~the assignment flow \eqref{eq:assignment-flow} is not a Riemannian gradient descent flow.
\end{theorem}
\begin{proof}
  By \eqref{eq:Rp-Pi0}, we have $\RO_{W}S(W) = \RO_{W}\PT S(W)$ and $\RO_W \colon \mc{T}_0 \to \mc{T}_0$ is a linear isomorphism (Lemma~\ref{lem:prop_exp}\eqref{enum:lem:prop_exp:RqIso}). Therefore, the question of existence of a potential $J \colon \mc{W} \to \R$ for the assignment flow \eqref{eq:assignment-flow} can be transferred to the Euclidean setting by applying $(\RO_W|_{\mc{T}_0})^{-1}$ to both sides of the equation $\RO_W S(W) = \Rgrad J(W)$, i.e.
  \begin{equation}
    \RO_{W}S(W) = -\Rgrad J(W) = -\RO_{W} \Egrad J(W) \qquad \Leftrightarrow \qquad \PT S(W) = -\PT\Egrad J(W) \in \mc{T}_0.
  \end{equation}
  If such a potential $J$ exists, then the negative Hessian of $J$ is given by
  \begin{equation}
    -\PT\Hess J(W) = d\big(-\PT\Egrad J\big)(W) = d(\PT\circ S)(W) = \PT dS(W) = dS(W),
  \end{equation}
  where the last equation follows from $d S(W)\colon \mc{T}_0 \to \mc{T}_0$. Furthermore, $\Hess J(W)$ and   therefore also $dS(W)$ must be symmetric with respect to the Euclidean scalar product on $\mc{T}_0$. Hence, in order to prove that a potential cannot exist, we show that $dS(W)$ is not symmetric at every point $W \in \mc{W}$. To this end, we construct a $W \in\mc{W}$ and $X \in \mc{T}_0$ with $dS(W)[X] - dS(W)^\T[X] \neq 0$. It suffices to show
\begin{equation}\label{eq:proof-asymmetry-i}
  dS_i(W)[X] - dS_i(W)^\T[X] \neq 0 \quad\text{for some row index}\quad i \in \mc{I}.
\end{equation}
To simplify notation, we write $D_{i}$ instead of $D_{\mc{F};i}$ in the remainder of the proof. Due to the hypothesis, we have
\begin{equation}\label{eq:Di-not-eins}
D_{i} = D_{\mc{F};i} \neq \R\eins.
\end{equation}
Let $k, l \in[c]$ be indices such that
\begin{equation}
D_{ik} = \min_{r\in[c]} D_{ir} \quad\text{and}\quad
D_{il} =  \max_{r\in[c]} D_{ir}.
\end{equation}
Relation \eqref{eq:Di-not-eins} implies $D_{ik} < D_{il}$ and
  \begin{equation}\label{eq:thm:nonex_of_potential:EQ1}
    e^{-\frac{1}{\rho} D_{ik}} > e^{-\frac{1}{\rho} D_{il}}.
  \end{equation}
Define
\begin{equation}\label{eq:proof-def-u}
u = e_k - e_l \in T_{0},\;\qquad e_{k},e_{l} \in \mc{B}_{c}.
\end{equation}
Since $c \geq 3$, there is also a point
\begin{equation}\label{eq:proof-choose-p}
p \in\mc{S} \quad\text{with}\quad
p \neq \BS \quad\text{and}\quad
p_k = p_l,
\end{equation}
e.g.~by choosing $0<\alpha<\tfrac{1}{c}$ and setting $p_k = p_l = \alpha$ and $p_r = (c-2)^{-1} (1 - 2\alpha)$ for $r \neq k, l$.

  With these choices, we define the point
\begin{equation}
W^p \in \mc{W},\qquad
W^p_j = \exp_p\Big(\frac{1}{\rho}D_j\Big) \quad\text{for all}\quad j \in \mc{I}.
\end{equation}
Also, set $v := \exp_{\BS}^{-1}(p)$. Then $W^p_j = \exp_{\BS}\big(v +\frac{1}{\rho}D_j\big)$ by \eqref{eq:expP_expressed_with_expQ} and Lemma~\ref{lem:Si_expressed_by_expBS} implies
\begin{subequations}\label{eq:proof-Si=p}
\begin{align}
  S_i(W) &= \exp_{\BS}\Big( \sum_{j \in \mc{N}(i)} \omega_{ij}\big(\exp_{\BS}^{-1}(W^p_j) - \frac{1}{\rho}D_j\big) \Big) \\
  &= \exp_{\BS}\Big( \sum_{j \in \mc{N}(i)} \omega_{ij}  v \Big) = \exp_c(v) = p,
  \end{align}
\end{subequations}
for all $i \in \mc{I}$.
Now, define
\begin{equation}
X^u \in \mc{T}_0 \quad\text{with}\quad
X^u_k = \begin{cases}
u \in T_0,&\text{if}\,k=i \\
X^u_j = 0,&\text{if}\,k\neq i.
\end{cases}
\end{equation}
Using the expressions for $dS_i(W^P)$ and $dS_i(W^p)^\T$ from Lemma~\ref{lem:dS_and_dST}, we obtain
\begin{subequations}
\begin{align}
  dS_i(W^p)[X^u] &- dS_i(W^p)^\T[X^u]
  \overset{\phantom{\eqref{eq:proof-Si=p}}}{=} \omega_{ii}\RO_{S_i(W^p)}\Big[ \frac{X^u_i}{W^p_i} \Big] - \omega_{ii}\PT\Big[\frac{\RO_{S_i(W^p)}X^u_i}{W^p_i}\Big]\\
  &\overset{\eqref{eq:proof-Si=p}}{=}
  \omega_{ii}\RO_{p}\Big[ \frac{u}{\exp_p(\frac{1}{\rho}D_i)} \Big] - \omega_{ii}\PT\Big[\frac{\RO_{p}u}{\exp_p(\frac{1}{\rho}D_i)}\Big]\\
  &\overset{\eqref{eq:lem:prop_exp:exp_explicitFormula}}{=} \omega_{ii} \la p, e^{\frac{1}{\rho}D_i}\ra \Big(\RO_{p}\Big[ \frac{u}{p} e^{-\frac{1}{\rho}D_i} \Big] - \PT\Big[ \frac{e^{-\frac{1}{\rho}D_i}}{p} \RO_{p}u\Big]\Big).
\end{align}
\end{subequations}
Since $\omega_{ii} \la p, e^{\frac{1}{\rho}D_i}\ra > 0$, we only have to check the expression inside the brackets. As for the first term, using \eqref{eq:def-Rp}, we have
  \begin{subequations}
  \begin{align}
    \RO_{p}\Big[ \frac{u}{p} e^{-\frac{1}{\rho}D_i} \Big] %= p \cdot \frac{1}{p} \cdot X_i\cdot e^{-\frac{1}{\rho}D_i} - \big\la \frac{1}{p} \cdot X_i\cdot e^{-\frac{1}{\rho}D_i} , p \big\ra p
    &= ue^{-\frac{1}{\rho}D_i} - \big\la u , e^{-\frac{1}{\rho}D_i} \big\ra p.
\intertext{
Setting $a := \big( \la e^{-\frac{1}{\rho}D_i}, \BS \ra \eins -  e^{-\frac{1}{\rho}D_i}\big)$, we obtain for the second term
}
    \PT\Big[ \frac{e^{-\frac{1}{\rho}D_i}}{p} \RO_{p}u\Big] &= \PT\Big[ e^{-\frac{1}{\rho}D_i} u - \la u, p\ra e^{-\frac{1}{\rho}D_i}\Big] = e^{-\frac{1}{\rho}D_i} u - \la u, e^{-\frac{1}{\rho}D_i}\ra \BS + \la u, p\ra a.
  \end{align}
  \end{subequations}
Thus, the term inside the brackets reads
  \begin{subequations}
  \begin{align}
    \RO_{p}\Big[ \frac{1}{p} ue^{-\frac{1}{\rho}D_i} \Big] - \PT\Big[ \frac{e^{-\frac{1}{\rho}D_i}}{p} \RO_{p}u\Big] &= - \big\la u , e^{-\frac{1}{\rho}D_i} \big\ra p + \la u, e^{-\frac{1}{\rho}D_i}\ra\BS - \la u, p\ra a\\
    &=\la u , e^{-\frac{1}{\rho}D_i} \ra \big(\BS -  p\big)  - \la u, p\ra a.
  \end{align}
  \end{subequations}
\eqref{eq:proof-def-u} and \eqref{eq:proof-choose-p} imply
  \begin{equation}
    \la u, e^{-\frac{1}{\rho}D_i}\ra = e^{-\frac{1}{\rho}D_{ik}} - e^{-\frac{1}{\rho}D_{il}} > 0\quad \text{and}\quad \la u, p\ra = p_k - p_l = 0
  \end{equation}
such that we can conclude
  \begin{equation}
  \la u , e^{-\frac{1}{\rho}D_i} \ra \big(\BS - p\big)  - \la u, p\ra a = (e^{-\frac{1}{\rho}D_{ik}} - e^{-\frac{1}{\rho}D_{il}}) \big(\BS - p\big) \neq 0.
  \end{equation}
This proves \eqref{eq:proof-asymmetry-i} and consequently the theorem.
\end{proof}
\subsection{S-parametrization}\label{sec:NovelRepAF:S-parametrization}

Even though Theorem \ref{thm:nonexistance_of_Potential_for_assignment_flow} says that no potential exists for the assignment flow in general,
we reveal in this section a `hidden' potential flow under an additional assumption. To this end, we decouple the assignment flow into two components and show that one component depends on the second one. The dominating second one, therefore, provides a new parametrization of the assignment flow. Assuming symmetry of the averaging matrix defined below by \eqref{eq:def-Omega-w}, the dominating flow becomes a Riemannian gradient descent flow. The corresponding potential defined on a continuous domain will be studied in subsequent sections.

For notational efficiency, we collect all weights \eqref{eq:weights-Omega-i} into the
\textit{averaging matrix}
\begin{equation}\label{eq:def-Omega-w}
  \Omega^{\w} \in \R^{n\times n}\; \text{with}\; \Omega^{\w}_{ij} := \psi_{\mc{N}_{i}}(j) \omega_{ij}
  = \begin{cases}                                             \omega_{ij} &\text{if } j \in \mc{N}_i,\\                                           0 &\text{else}
\end{cases}, \quad \text{for } i, j \in \mc{I}.
\end{equation}
$\Omega^{\w}$ encodes the spatial structure of the graph and the weights. For an arbitrary matrix $M \in \R^{n\times c}$, the average of its row vectors using the weights indexed by the neighborhood $\mc{N}_{i}$ is given by
\begin{equation}\label{eq:wik-by-Omega}
\sum_{k \in \mc{N}_i} \omega_{ik} M_k
= \sum_{k \in \mc{I}} \Omega^{\w}_{ik}M_k
= M^{\T} \Omega^{\w}_{i}.
\end{equation}
Thus, all row vector averages are given as row vectors of the matrix $\Omega^{\w} M$.

We now introduce a new representation of the assignment flow.
\begin{proposition}\label{prop:decoupled-flow}
  The assignment flow \eqref{eq:assignment-flow} is equivalent to the system
  \begin{subequations}\label{eq:decoupled-flow}
  \begin{align}
    \dot{W} &= \RO_{W} \ol{S} &&\text{with}\quad W(0) = \BW\label{eq:def-W-flow}\\
    \dot{\ol{S}} &= \RO_{\ol{S}} [\Omega\ol{S}] &&\text{with}\quad\ol{S}(0) = S(\BW).\label{eq:def-S-flow}
  \end{align}
  \end{subequations}
\end{proposition}
\begin{remark}\label{rem:decoupled-flow}
We observe that the flow $W(t)$ is completely determined by $S(t)$. In the following, we refer to the dominating part \eqref{eq:def-S-flow} as the \textit{$S$-flow}.
\end{remark}
\begin{proof}
  Let $W(t)$ be a solution of the assignment flow, i.e. $\dot{W}_i = \RO_{W_i} S_i(W)$ for all $i \in \mc{I}$. Set $\ol{S}(t) := S(W(t))$. Then \eqref{eq:def-W-flow} is immediate from the assumption on $W$. Using the expression for $dS_i(W)$ from Lemma~\ref{lem:dS_and_dST} gives
  \begin{equation}\label{eq:prop:decoupled-flow_EQ1}
    \dot{\ol{S}}_i = \frac{d}{dt}S(W)_i = dS_i(W)[\dot{W}] = \sum_{j\in\mc{N}_i} \omega_{ij}\RO_{S_i(W)}\Big[\frac{\dot{W}_j}{W_j}\Big].
  \end{equation}
  Since $W$ solves the assignment flow and $\RO_{S_i(W)} = \RO_{S_i(W)}\PT$ by \eqref{eq:Rp-Pi0} with $\ker(\PT) = \R\eins_c$, it follows using the explicit expresssion \eqref{eq:def-Rp} of $\RO_{S_i(W)}$ that
  \begin{subequations}
  \begin{align}
    \RO_{S_i(W)}\Big[\frac{\dot{W}_j}{W_j}\Big] &= \RO_{S_i(W)}\Big[\frac{\RO_{W_j}S_j(W)}{W_j}\Big] = \RO_{S_i(W)}\Big[S_j(W) - \la W_j, S_j(W)\ra \eins_c\Big]\\
    &= \RO_{S_i(W)}\big[S_j(W)\big].
  \end{align}
  \end{subequations}
  Back-substitution of this identity into \eqref{eq:prop:decoupled-flow_EQ1}, pulling the linear map $\RO_{S_i(W)}$ out of the sum and keeping $S_i(W) = \ol{S}_i$ in mind, results in
  \begin{equation}
    \dot{\ol{S}}_i = \RO_{S_i(W)}\big[S_j(W)\big]
    = \RO_{\ol{S}_i}\Big[ \sum_{j\in\mc{N}_i} \omega_{ij} \ol{S}_j\Big] = \RO_{\ol{S}_i} [\ol{S}^{\T}\Omega^{\w}_{i}] \quad \text{for all } i \in \mc{I}.%\qedhere
  \end{equation}
Collecting these vectors as row vectors of the matrix $\dot{\ol{S}}$ gives \eqref{eq:def-S-flow}.
\end{proof}
\begin{remark}\label{rem:S-flow}
Henceforth, we write $S$ for the $S$-flow $\ol{S}$ to stress the underlying connection to the assignment flow and to simplify the notation.
\end{remark}
We next show that the $S$-flow which essentially determines the assignment flow (Remark \ref{rem:S-flow}) becomes a Riemannian descent flow under the additional assumption that the averaging matrix \eqref{eq:def-Omega-w} is symmetric.

\begin{proposition}\label{prop:S-flow-gradient}
  Suppose the weights defining the similarity map in \eqref{eq:def-Si} are symmetric, i.e. $(\Omega^{\w})^\T = \Omega^{\w}$. Then the $S$-flow \eqref{eq:def-S-flow} is a Riemannian gradient decent flow $\dot{S} = -\Rgrad J(S)$, induced by the potential
  \begin{equation}\label{eq:def_J_Sflow}
    J(S) := -\frac{1}{2}\la S, \Omega^{\w} S\ra,\quad S \in \mc{W}.
  \end{equation}
\end{proposition}
\begin{proof}
  Let $\gamma \colon (-\veps, \veps) \to \mc{W},\,\veps > 0$, be any smooth curve with $\dot{\gamma}(0) = V \in \R^{n\times c}$ and $\gamma(0) = S$. By the symmetry of $\Omega^{\w}$, we have $\la \Egrad J(S), V\ra = dJ(S)[V] = \frac{d}{dt}J(\gamma(t))\big|_{t = 0} = -\la \Omega^{\w} S, V\ra$ for all $V \in \R^{n\times c}$. Therefore, $\Egrad J(S) = - \Omega^{\w} S$. Thus, the Riemannian gradient is given by $\Rgrad J(S) = \RO_S[\Egrad J(S)] = - \RO_S[\Omega^{\w} S]$.
\end{proof}

Consider
\begin{equation}\label{eq:def-LmcG}
L_{\mc{G}} = I_{n} - \Omega^{\w},
\end{equation}
where $I_{n} \in \R^{n\times n}$ is the identity matrix. Since $I_{n} = \Diag(\Omega^{\w} \eins_{n})$ by \eqref{eq:weights-Omega-i} is the degree matrix of the symmetric averaging matrix $\Omega^{\w}$, $L_{\mc{G}}$ can be regarded as Laplacian (matrix) of the underlying  undirected weighted graph $\mc{G} = (\mc{V}, \mc{E})$\footnote{For undirected graphs, the graph Laplacian is commonly defined by the weighted \textit{adjacency} matrices with diagonal entries $0$, whereas $\Omega^{\w}_{ii}=\w_{ii}>0$. The diagonal entries do not affect the quadratic form \eqref{eq:J_expressed_as_dirichlet_energy_Sflow}, however.}. For the analysis of the $S$-flow it will be convenient to rewrite the potential \eqref{eq:def_J_Sflow} accordingly.
\begin{proposition}\label{prop:potential-S-flow}
  Under the assumption of Proposition \ref{prop:S-flow-gradient}, the potential \eqref{eq:def_J_Sflow} can be written in the form
  \begin{equation}\label{eq:J_expressed_as_dirichlet_energy_Sflow}
    J(S) = \frac{1}{2}\la S, L_{\mc{G}} S\ra - \frac{1}{2}\| S\|^2 = \frac{1}{4} \sum_{i\in\mc{I}} \sum_{j\in\mc{N}_i} \omega_{ij}\|S_i-S_j\|^2 - \frac{1}{2}\|S\|^2.
  \end{equation}
The matrix $L_{\mc{G}}$ is symmetric, positive semidefinite and $L_{\mc{G}}\eins_{n}=0$.
\end{proposition}
\begin{proof}
We have $J(S)=-\frac{1}{2}\la S,(\Omega^{\w}-I_{n}) S\ra + \la S,S\ra = \frac{1}{2}(\la S,L_{\mc{G}} S\ra-\|S\|^{2})$. Thus, we focus on the sum of \eqref{eq:J_expressed_as_dirichlet_energy_Sflow}.

  First, note that $\| S_j - S_i\|_2^2 = \la S_j, S_j - S_i\ra + \la S_i, S_i - S_j\ra$. Since $\psi_{\mc{N}_i}(j) = \psi_{\mc{N}_j}(i)$ and $\omega_{ij} = \omega_{ji}$ by assumption, we have
  \begin{subequations}
  \begin{align}
    \sum_{i \in \mc{I}} \sum_{j\in\mc{N}_i} \omega_{ij} \la S_j,  S_j - S_i\ra &=
      \sum_{i, j\in\mc{I}} \psi_{\mc{N}_{i}}(j) \omega_{ij}\la S_j,  S_j - S_i\ra
      = \sum_{i, j\in\mc{I}} \psi_{\mc{N}_{j}}(i) \omega_{ji}\la S_j,  S_j - S_i\ra\\
    &= \sum_{j\in\mc{I}} \sum_{i\in\mc{N}_j} \omega_{ji}\la S_j, S_j - S_i\ra
      = \sum_{i\in\mc{I}} \sum_{j\in\mc{N}_i} \omega_{ij}\la S_i, S_i - S_j\ra,
  \end{align}
  \end{subequations}
  where the last equality follows by renaming the indices $i$ and $j$. Thus, using \eqref{eq:weights-Omega-i},
  \begin{subequations}\label{eq:SLS-explicit}
  \begin{align}
    \sum_{i \in \mc{I}} \sum_{j\in\mc{N}_i} \omega_{ij} \| S_i - S_j\|_2^2 &= \sum_{i \in \mc{I}} \sum_{j\in\mc{N}_i} \omega_{ij} \la S_j,  S_j - S_i\ra + \sum_{i\in\mc{I}} \sum_{j\in\mc{N}_i} \omega_{ij}\la S_i, S_i - S_j\ra\\
    &= 2 \sum_{i\in\mc{I}} \sum_{j\in\mc{N}_i} \omega_{ij}\la S_i, S_i - S_j\ra
      = 2 \sum_{i\in\mc{I}} \Big\la S_i, S_i - \sum_{j\in\mc{N}_i} \omega_{ij} S_j\Big\ra\\
    &= 2 \sum_{i\in\mc{I}} \la S_i, (L S)_{i}\ra = 2 \la S, L S\ra.
  \end{align}
  \end{subequations}
The properties of $L_{\mc{G}}$ follow from the symmetry of $\Omega^{\w}$, nonnegativity of the quadratic form \eqref{eq:SLS-explicit} and definition \eqref{eq:def-LmcG}.
\end{proof}

% !TEX root =  ../continuous-domain_assignment_flows.tex
%%%%%%%%%%%%%%%%%%%%%%%%%%%%%%%%%%%%%%%%%%

\section{Continuous-Domain Variational Approach}\label{sec:zero-scale-functional}

In this section, we study a continuous-domain variational formulation of the potential of Proposition \ref{prop:potential-S-flow}. We confine ourselves to uniform weights \eqref{eq:weights-Omega-i} and neighborhoods \eqref{eq:def-Ni} that only contain the nearest neighbors of each vertex $i$, such that $L_{\mc{G}}$ becomes the discretized ordinary Laplacian. As a result, we consider the problem to minimize the functional
\begin{subequations}\label{eq:def-E-alpha}
\begin{align}
&E_{\alpha} \colon H^{1}(\mc{M};\R^{c}) \to \R,
\\
&E_{\alpha}(S) := \int_{\mc{M}}\|DS(x)\|^{2}-\alpha \|S(x)\|^{2} \dd{x},\qquad \alpha > 0.
\end{align}
\end{subequations}
Throughout this section, $\mc{M} \subset \R^{2}$ is a simply-connected bounded open subset in the Euclidean plane. Parameter $\alpha$ controls the interaction between regularization and enforcing integrality when $S(x),\, x \in \mc{M}$ is restricted to values in the probability simplex.

We prove well-posedness for vanishing (Section \ref{sec:Well-Posedness}) and Dirichlet boundary conditions (Section \ref{sec:Fixed Boundary Conditions}), respectively, and specify explicitly the set of minimizers in the former case. The gradient descent flow corresponding to the latter case, initialized by means of given data and with parameter value $\alpha=1$, may be seen as continuous-domain extension of the assignment flow, that is parametrized according to \eqref{prop:decoupled-flow} and operates at the smallest spatial scale in terms of the size $|\mc{N}_{i}|$ of uniform neighborhoods \eqref{eq:def-Ni} (in the discrete formulation \eqref{eq:assignment-flow}: nearest neighbor averaging). We illustrate this by a numerical example (Section \ref{sec:Numerical-Example}), based on discretizing \eqref{eq:def-E-alpha} and applying an algorithm that mimics the $S$-flow and converges to a local minimum of the non-convex functional \eqref{eq:def-E-alpha}, by solving a sequence of convex programs.

We point out that $\mc{M}$ could be turned into a Riemannian manifold using a metric that reflects images features (edges etc.), as was proposed with the Laplace-Beltrami framework for image denoising \cite{Kimmel:2000aa}. In this work we focus on the essential point, however, that distinguishes image \textit{denoising} from image \textit{labeling}, i.e.~the interaction of the two terms \eqref{eq:def-E-alpha} that essentially is a consequence of the information geometry of the assignment manifold $\mc{W}$ \eqref{schnoerr-eq:def-mcW}.

%%%
\subsection{Well-Posedness}\label{sec:Well-Posedness}

Based on \eqref{eq:def-Delta-n} we define the closed convex set
\begin{equation}
  \ms{D}^{1}(\mc{M}) = \{ S\in H^{1}(\mc{M}; \R^c)\colon\ S(x) \in \Delta_{c}\;\text{a.e.~in}\; \mc{M} \}.
\end{equation}

and focus on the variational problem
\begin{equation}\label{eq:def-Ealpha-D1M}
\inf_{S \in \ms{D}^{1}(\mc{M})} E_{\alpha}(S),
\end{equation}
with $E_{\alpha}$ given by \eqref{eq:def-E-alpha}. $E_{\alpha}$ is smooth but nonconvex. We specify the set of minimizers (Prop.~\ref{prop:E-alpha-inf}). Recall notation \eqref{eq:def-e1-en}.
\begin{lemma}\label{lem:p-mcS-norm}
Let $p \in \Delta^{c}$. Then $\|p\|=1$ if and only if $p \in \mc{B}_{c}$.
\end{lemma}
\begin{proof}
The `if' statement is obvious. As for the `only if', suppose $p \not\in \mc{B}_{c}$, i.e.~$p_{i} < 1$ for all $i \in [c]$. Then $p_{i}^{2} < p_{i}$ and $\|p\|^{2} < \|p\|_{1}=1$.
\end{proof}
\begin{proposition}\label{prop:E-alpha-inf}
The functional $E_{\alpha} \colon \ms{D}^{1}(\mc{M}) \to \R$ given by \eqref{eq:def-E-alpha} is lower bounded,
\begin{equation}\label{eq:E-alpha-inf}
E_{\alpha}(S) \geq -\alpha \Vol(\mc{M}) > -\infty,\qquad \forall S \in \ms{D}^{1}(\mc{M}).
\end{equation}
This lower bound is attained at some point in
\begin{equation}\label{eq:E-alpha-minimizers}
\operatornamewithlimits{\arg \min}_{S \in \ms{D}^{1}(\mc{M})} E_{\alpha}(S) = \begin{cases}
\{S_{e_{1}},\dotsc,S_{e_{c}}\}, &\text{if}\;\alpha > 0,
\\
\{S_{p} \colon \mc{M} \to \Delta \colon p \in \Delta\},
&\text{if}\;\alpha=0,
\end{cases}
\end{equation}
where, for any $p \in \Delta$, $S_{p}$ denotes the constant vector field $x \mapsto S_{p}(x)=p$.
\end{proposition}
\begin{proof}
Let $p \in \Delta$. Then $\|p\|^{2} \leq \|p\|_{1}=1$. It follows for $S \in \ms{D}^{1}(\mc{M})$ that
\begin{equation}
E_{\alpha}(S) \geq -\alpha \|S\|_{\mc{M}}^{2} \geq -\alpha \|1\|_{\mc{M}} = -\alpha\Vol(\mc{M}),
\end{equation}
which is \eqref{eq:E-alpha-inf}.

We next show that the right-hand side of \eqref{eq:E-alpha-minimizers} specifies minimizers of $E_{\alpha}$. For any $p \in \Delta$, the constant vector field $S_{p}$ is contained in $\ms{D}^{1}(\mc{M})$. Consider specifically $S_{e_{i}},\,i \in [c]$. Since $\|S_{e_{i}}(x)\|=\|e_{i}\|=1$ and $DS_{e_{i}} \equiv 0$, the lower bound is attained, $E_{\alpha}(S_{e_{i}})=-\alpha\Vol(\mc{M})$, and the functions $\{S_{e_{1}},\dotsc,S_{e_{c}}\}$ minimize $E_{\alpha}$, for every $\alpha \geq 0$. If $\alpha=0$, then the constant functions $S_{p}$ are minimizers as well, for any $p \in \Delta$, since then
\begin{equation}\label{eq:proof-fp-minimizer}
E_{\alpha}(S_{p}) = \|DS_{p}\|_{\mc{M}}^{2} = 0 = - 0 \cdot \Vol(\mc{M}).
\end{equation}

We conclude by showing that no minimizers other than \eqref{eq:E-alpha-minimizers} exist. Let $S_{\ast} \in \ms{D}^{1}(\mc{M})$ be another minimizer of $E_{\alpha}$ with $E_{\alpha}(S_{\ast})=-\alpha\Vol(\mc{M})$. We distinguish the two cases $\alpha=0$ and $\alpha >0$.

If $\alpha=0$, then $S_{\ast}$ satisfies \eqref{eq:proof-fp-minimizer} and $\|DS_{\ast}\|_{\mc{M}}^{2}=0$. Since $\|DS_{\ast;i}\|_{\mc{M}}\leq \|DS_{\ast}\|_{\mc{M}}=0$ for every $i \in [c]$, $S_{\ast}$ is constant by Lemma
\ref{lem:Du-0-constant}, i.e.~a $p \in \Delta$ exists such that $S_{\ast}=S_{p}$ a.e.

If $\alpha > 0$, then using the equation $E_{\alpha}(S_{\ast})=-\alpha\Vol(\mc{M})$ and $\|S_{\ast}(x)\|^{2}\leq 1$ gives
\begin{subequations}
\begin{align}
\alpha\Vol(\mc{M}) &\leq \|DS_{\ast}\|_{\mc{M}}^{2} + \alpha\Vol(\mc{M}) = \|DS_{\ast}\|_{1;\mc{M}}^{2} - E_{\alpha}(S_{\ast}) = \alpha\|S_{\ast}\|_{\mc{M}}^{2}
\\
&\leq \alpha \|1\|_{\mc{M}} = \alpha\Vol(\mc{M}),
\end{align}
\end{subequations}
which shows $\|DS_{\ast}\|_{\mc{M}}=0$ and hence by Lemma
\ref{lem:Du-0-constant} again $S_{\ast}=S_{p}$ for some $p \in \Delta$. The preceding inequalities also imply $\Vol(\mc{M})=\|S_{\ast}\|_{\mc{M}}^{2}$, i.e.~$\|S_{\ast}(x)\|=1$ a.e. By Lemma \ref{lem:p-mcS-norm}, we conclude $S_{\ast}=S_{p}$ with $p \in \mc{B}_{c}$, that is $S_{\ast} \in \{S_{e_{1}},\dotsc,S_{e_{c}}\}$.
\end{proof}
Proposition \ref{prop:E-alpha-inf} highlights the effect of the concave term of the objective $E_{\alpha}$ \eqref{eq:def-E-alpha}: labelings are enforced in the absence of data. Below, the latter are taken into account (i) by imposing non-zero boundary conditions and (ii) by initalizing a corresponding gradient flow (Section \ref{sec:Numerical-Example}).

%%%
\subsection{Fixed Boundary Conditions}
\label{sec:Fixed Boundary Conditions}

In this section, we consider the case where boundary conditions are imposed by restricting the feasible set of problem \eqref{eq:def-Ealpha-D1M} to
\begin{equation}\label{eq:def-Ag1M}
\mc{A}_{g}^{1}(\mc{M})
= \{S \in \mc{D}^{1}(\mc{M}) \colon S-g \in H_{0}^{1}(\mc{M};\R^{c})\}
= \big(g + H_{0}^{1}(\mc{M};\R^{c})\big) \cap \mc{D}^{1}(\mc{M})
\end{equation}
for some fixed $g$ that prescribes simplex-valued boundary values (in the trace sense).
As intersection of a closed affine subspace and a closed convex set, $\mc{A}_{g}^{1}(\mc{M})$ is closed convex.

Weak lower semicontinuity is a key property for proving the existence of minimizers. In the case of $E_{\alpha}$ \eqref{eq:def-E-alpha} this is not immediate, due to the lack of convexity.
\begin{proposition}\label{prop:Ealpha-wlsc}
The functional $E_{\alpha}$ given by \eqref{eq:def-E-alpha} is weak sequentially lower semicontinuous on $\mc{A}_{g}^{1}(\mc{M})$, i.e.~for any sequence $(S_{n})_{n \in \N} \subset \mc{A}_{g}^{1}(\mc{M})$ weakly converging to $S \in \mc{A}_{g}^{1}(\mc{M})$, the inequality
\begin{equation}
E_{\alpha}(S) \leq \liminf_{n \to \infty} E_{\alpha}(S_{n})
\end{equation}
holds.
\end{proposition}
\begin{proof}
Let $S_{n} \rightharpoonup S$ converge weakly in $\mc{A}_{g}^{1}(\mc{M}) \subset H_{0}^{1}(\mc{M};\R^{c})$. Then, by Prop.~\ref{prop:properties-Banach}(c),
\begin{equation}\label{eq:proof-fnorm-wlsc}
\|S\|_{1;\mc{M}} \leq \liminf_{n \to \infty} \|S_{n}\|_{1;\mc{M}}.
\end{equation}
Since $S,S_{n} \in \mc{A}_{g}^{1}(\mc{M})$, we also have $(S_{n}-g) \rightharpoonup (S-g)$ in $H_{0}^{1}(\mc{M};\R^{c})$ by \eqref{eq:def-Ag1M} and consequently $S_{n}\to S$ strongly in $L^{2}(\mc{M};\R^{c})$ due to \eqref{eq:compact-embedding}. Taking into account \eqref{eq:proof-fnorm-wlsc} and $\liminf_{n \to \infty} \|s_{n}\|_{\mc{M}}=\lim_{n \to \infty}\|S_{n}\|_{\mc{M}} = \|S\|_{\mc{M}}$, we obtain
\begin{subequations}
\begin{align}
E_{\alpha}(S)
&= \|S\|_{1;\mc{M}}^{2}-(1+\alpha)\|S\|_{\mc{M}}^{2}
\leq \liminf_{n \to \infty} \|S_{n}\|_{1;\mc{M}}^{2} + \liminf_{n \to \infty} \big(-(1+\alpha)\|S_{n}\|_{\mc{M}}^{2}\big)
\\
&\leq \liminf_{n \to \infty} E_{\alpha}(S_{n}).
\end{align}
\end{subequations}
\end{proof}
We are now prepared to show that $E_{\alpha}$ attains its minimal value on $\mc{A}_{g}^{1}(\mc{M})$, following the basic proof pattern of \cite[Ch.~38]{Zeidler:1985aa}.
\begin{theorem}\label{thm:inf-Ealpha-A1M}
Let $E_{\alpha}$ be given by \eqref{eq:def-E-alpha}. There exists $S_{\ast} \in \mc{A}_{g}^{1}(\mc{M})$ such that
\begin{equation}\label{eq:Ealpha-mcA1g}
E_{\alpha}^{\ast} := E_{\alpha}(S_{\ast}) = \inf_{S \in \mc{A}_{g}^{1}(\mc{M})} E_{\alpha}(S).
\end{equation}
\end{theorem}
\begin{proof}
Let $(S_{n})_{n \in \N} \subset \mc{A}_{g}^{1}(\mc{M})$ be a minimizing sequence such that
\begin{equation}\label{eq:proof-inf-Ealpha-A1M-fn}
\lim_{n \to \infty} E_{\alpha}(S_{n})=E_{\alpha}^{\ast}.
\end{equation}
Then there exists some sufficiently large $n_{0} \in \N$ such that
\begin{equation}
1+E_{\alpha}^{\ast} \geq E_{\alpha}(S_{n})
= \|S_{n}\|_{1;\mc{M}}^{2}-(1+\alpha)\|S_{n}\|_{\mc{M}}^{2},\qquad \forall n \geq n_{0}.
\end{equation}
Since $S_{n}(x) \in \Delta$ for a.e.~$x \in \mc{M}$, we have $\|S_{n}\|_{\mc{M}}^{2} \leq \Vol(\mc{M})$ and hence obtain
\begin{equation}
\|S_{n}\|_{1;\mc{M}}^{2}
\leq 1+E_{\alpha}^{\ast} + (1+\alpha)\|S_{n}\|_{\mc{M}}^{2}
\leq 1+E_{\alpha}^{\ast} + (1+\alpha)\Vol(\mc{M}),\qquad
\forall n \geq n_{0}.
\end{equation}
Thus the sequence $(S_{n})_{n \in \N} \subset H^{1}(\mc{M};\R^{c})$ is bounded and, by Prop.~\ref{prop:properties-Banach}(c), we may extract a weakly converging subsequence $S_{n_{k}} \rightharpoonup S_{\ast} \in H^{1}(\mc{M};\R^{c})$. Since $\mc{A}_{g}^{1}(\mc{M}) \subset H^{1}(\mc{M};\R^{c})$ is closed convex, Prop.~\ref{prop:properties-Banach}(a) implies $S_{\ast} \in \mc{A}_{g}^{1}(\mc{M})$. Consequently, by Prop.~\ref{prop:Ealpha-wlsc} and \eqref{eq:proof-inf-Ealpha-A1M-fn},
\begin{equation}
E_{\alpha}(S_{\ast}) \leq \liminf_{k \to \infty} E_{\alpha}(S_{n_{k}}) = \lim_{k \to \infty} E_{\alpha}(S_{n_{k}}) = E_{\alpha}^{\ast}
\end{equation}
which implies $E_{\alpha}(S_{\ast})=E_{\alpha}^{\ast}$, i.e.~$S_{\ast} \in \mc{A}_{g}^{1}(\mc{M})$ minimizes $E_{\alpha}$.
\end{proof}
\subsection{Numerical Algorithm and Example}\label{sec:Numerical-Example}
We consider the variational problem \eqref{eq:Ealpha-mcA1g}
\begin{equation}\label{eq:problem-for-numerics}
\inf_{S \in \mc{A}_{g}^{1}(\mc{M})} \int_{\mc{M}}\|D S\|^{2} - \alpha \|S\|^{2} dx,
\end{equation}
for some fixed $g$ specifying the boundary values $S|_{\partial\mc{M}} = g|_{\partial\mc{M}}$, and the problem to compute a local minimum numerically using an optimization scheme that mimics the $S$-flow of Proposition \ref{prop:decoupled-flow}.

Based on \eqref{eq:def-Ag1M}, we rewrite the problem in the form
\begin{subequations}
\begin{align}
&\inf_{f \in H_{0}^{1}(\mc{M};\R^{c})}\Big\{\|D (g+f)\|_{\mc{M}}^{2} - \alpha \|g+f\|_{\mc{M}}^{2} dx + \delta_{\mc{D}^{1}(\mc{M})}(g+f)\Big\}
\\
&= \inf_{f \in H_{0}^{1}(\mc{M};\R^{c})}\Big\{\|D f\|_{\mc{M}}^{2} + 2\la D g, D f \ra_{\mc{M}} - \alpha\big(\|f\|_{\mc{M}}^{2}+2\la g,f\ra_{\mc{M}}\big)+ \delta_{\mc{D}^{1}(\mc{M})}(g+f)\Big\} + C,
\end{align}
\end{subequations}
where the last constant $C$ collects terms not depending on $f$. We discretize the problem as follows. $f$ becomes a vector $f \in \R^{c\,n}$ with $n = |\mc{I}|$ subvectors $f_{i} \in \R^{c},\,i \in \mc{I}$ or alternatively with $c = |\mc{J}|$ subvectors $f^{j},\,j \in \mc{J}$. The inner product $\la g, f \ra_{\mc{M}}$ is replaced by $\la g, f \ra = \sum_{i \in [n]} \la g_{i},f_{i}\ra = \sum_{j \in [c]} \la g^{j},f^{j}\ra$. We keep the symbols $f, g$ for simplicity and indicate the discretized setting by the subscript $n$ as introduced next.

$D$ becomes a gradient matrix $D_{n}$ that estimates the gradient of each subvector $f^{j}$ separately, such that
\begin{equation}
L_{n} f := D_{n}^{\T} D_{n} f
\end{equation}
is the basic discrete 5-point stencil Laplacian  applied to each subvector $f^{j}$. The feasible set $\mc{D}^{1}(\mc{M})$ is replaced by the closed convex set
\begin{equation}\label{eq:def-mcDn}
\mc{D}_{n} := \{f \geq 0 \colon \la \eins_{c},f_{i}\ra=1,\; \forall i \in \mc{I}\}.
\end{equation}
Thus the discretized problem reads
\begin{equation}\label{eq:problem-discretized}
\inf_{f}\Big\{\|D_{n} f\|^{2} + 2\la L_{n} g-\alpha g,f\ra - \alpha\|f\|^{2} + \delta_{\mc{D}_{n}}(g+f)\Big\}.
\end{equation}
Having computed a local minimum $f_{\ast}$, the corresponding local minimum of \eqref{eq:problem-for-numerics} is $S_{\ast}=g+f_{\ast}$.

In order to compute $f_{\ast}$, we applied the proximal forward-backward scheme
\begin{equation}\label{eq:fk-PALM}
f^{(k+1)}
= \arg\min_{f}\Big\{
\|D_{n} f\|^{2}+2\la L_{n} g-\alpha (g+f^{(k)}),f\ra + \frac{1}{2\tau_{k}}\|f-f^{(k)}\|^{2} + \delta_{\mc{D}_{n}}(g+f)\Big\}
,\quad k \geq 0,
\end{equation}
with proximal parameters $\tau_{k},\,k \in \N$ and initialization $f_{i}^{(0)},\, i \in \mc{I}$ specified further below. The iterative scheme \eqref{eq:fk-PALM} is a special case of the PALM algorithm \cite[Sec.~3.7]{Bolte:2014aa}.
Ignoring the proximal term, each problem \eqref{eq:fk-PALM} amounts to solve $c$ (discretized) Dirichlet problems with the boundary values of $g^{j},\, j \in [c]$ imposed, and with right-hand sides that change during the iteration since they depend on $f^{(k)}$. The solutions $(f^{j})^{(k)},\, j \in \mc{J}$ to these Dirichlet problems depend on each other, however, through the feasible set \eqref{eq:def-mcDn}. At each iteration $k$, problem \eqref{eq:fk-PALM} can be solved by convex programming. The proximal parameters $\tau_{k}$ act as stepsizes such that the sequence $f^{(k)}$ does not approach a local minimum too rapidly. Then the interplay between the linear form that adapts during the iteration and the regularizing effect of the Laplacians can find a labeling (partition) corresponding to a good local optimum.

As for $g$, we chose $g_{i}=L_{i}(\eins_{\mc{S}}),\,i \in \mc{I}$ at boundary vertices $i$ and $g_{i} =0$ at every interior vertex $i$. Consequently, with the initialization $f_{i}^{(0)}=L_{i}(\eins_{\mc{S}}),\,i \in \mc{I}$ at interior vertices (the boundary values of $f$ are zero), the sequence $S^{(k)}=g+f^{(k)}$  mimics the $S$-flow of Proposition \ref{prop:decoupled-flow} where the given data also show up in the initialization $\ol{S}(0)$ only.

Figure \ref{fig:31-colors} provides an illustration using an experiment adopted from \cite[Fig.~6]{Astrom:2017ac}, originally designed to evaluate the performance of geometric regularization of label assignments through the assignment flow in an unbiased way. Parameter values are specified in the caption. The result confirms that the continuous-domain formulations discussed above represent the assignment flow at the smallest spatial scale.
\begin{figure}
\centerline{
\includegraphics[width=0.24\textwidth]{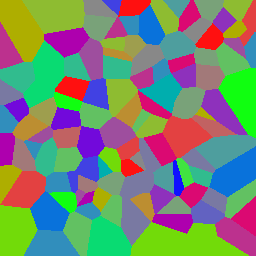}\hfill
\includegraphics[width=0.24\textwidth]{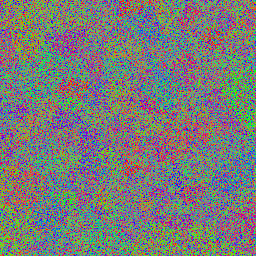}\hfill
\includegraphics[width=0.24\textwidth]{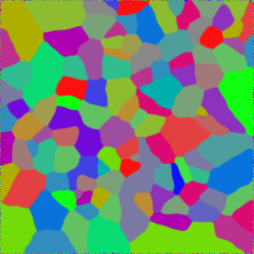}\hfill
\includegraphics[width=0.24\textwidth]{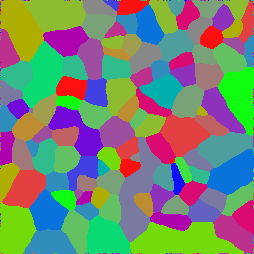}
}
\vspace{0.005\textwidth}
\centerline{
\includegraphics[width=0.24\textwidth]{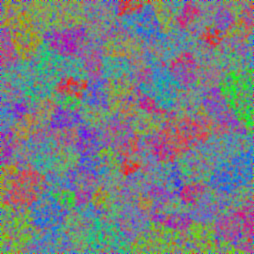}\hfill
\includegraphics[width=0.24\textwidth]{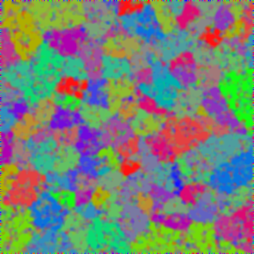}\hfill
\includegraphics[width=0.24\textwidth]{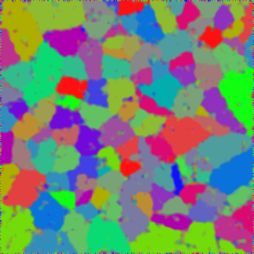}\hfill
\includegraphics[width=0.24\textwidth]{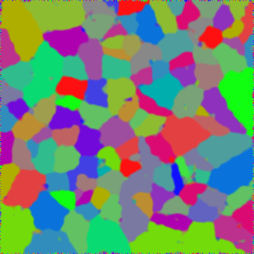}
}
\caption{
Evaluation of the numerical scheme \eqref{eq:fk-PALM} that mimics the $S$-flow of Proposition \ref{prop:decoupled-flow}. Parameter values: $\alpha=1, \tau_{k}=\tau=10,\, \forall k$. \textbf{Top}, from left to right: Ground truth, noisy input data $f^{(0)}$, iterate $f^{(100)}$ and $f_{\ast}$ resulting from $f^{(100)}$ by a trivial rounding step. $S^{(k)}=f^{(k)}+g$ differs from $f^{(k)}$ by the boundary values corresponding to the noisy input data. Inspecting the values of $f^{(100)}$ close to the boundary shows that the influence of boundary noise is minimal. \textbf{Bottom}, from left to right: The iterates $f^{(10)}, f^{(20)}, f^{(30)}, f^{(40)}$. Taking into account rounding as post-processing step, the sequence $f^{(k)}$ quickly converges after rounding to a reasonable partition. About $50$ more iterations are required to fix the values at merely few hundred remaining pixels. Slight rounding of the geometry of the components of the partition, in comparison to ground truth, corresponds to using uniform weights \eqref{eq:weights-Omega-i} for the assignment flow.
}
\label{fig:31-colors}
\end{figure}
%

% !TEX root =  ../continuous-domain_assignment_flows.tex
%%%%%%%%%%%%%%%%%%%%%%%%%%%%%%%%%%%%%%%%%%

\subsection{A PDE Characterizing Optimal Assignment Flows}\label{sec:VI-PDE}

\begin{proposition}\label{prop:VI}
Let $S_{\ast}$ solve the variational problem \eqref{eq:problem-for-numerics}. Then $S_{\ast}$ satisfies the variational inequality
\begin{equation}\label{eq:VI-Sast}
\la DS_{\ast},DS-DS_{\ast}\ra_{\mc{M}} - \alpha \la S_{\ast},S-S_{\ast}\ra_{\mc{M}} \geq 0,\qquad \forall S \in \mc{A}_{g}^{1}(\mc{M}).
\end{equation}
\end{proposition}
\begin{proof}
Functional $E_{\alpha}$ given by \eqref{eq:problem-for-numerics} is Gateaux-differentiable with derivative
\begin{equation}
\la E'(\alpha)(S_{\ast}),S\ra_{H^{-1}(\mc{M};\R^{c})\times H_{0}^{1}(\mc{M};\R^{c})} = 2\big(\la DS_{\ast},S\ra_{\mc{M}} - \alpha\la S_{\ast}, S\ra_{\mc{M}}\big).
\end{equation}
The assertion follows from applying Theorem \ref{thm:VI-necessarily}.
\end{proof}
We conclude this section by deriving a PDE corresponding to \eqref{eq:VI-Sast}, that a minimizer $S_{\ast}$ is supposed to satisfy in the weak sense. The derivation is \textit{formal} in that we adopt the unrealistic regularity assumption
\begin{equation}\label{eq:S-regularity-assumption}
S_{\ast} \in \mc{A}_{g}^{2}(\mc{M}),
\end{equation}
with $\mc{A}_{g}^{2}(\mc{M})$ defined analogous to \eqref{eq:def-Ag1M}. While this will hold for the continuous-domain \textit{linear} problems corresponding to \eqref{eq:fk-PALM} at each step $k$ of the iteration  and for sufficiently smooth $\partial\mc{M}$, it will not hold in the limit $k \to \infty$, since we expect (and wish) $S_{\ast}$ to become discontinuous, contrary to the regularity assumption \eqref{eq:S-regularity-assumption} and the continuity implied by the Sobolev embedding theorem for $\mc{M} \subset \R^{d}$ with $d=2$.
Nevertheless, since the PDE provides another interpretation of the assignment flow, we state it -- see \eqref{eq:AF-PDE} below -- and hope it will stimulate further research.

In view of the assumption \eqref{eq:S-regularity-assumption}, set
\begin{equation}
S_{\ast} = g + f_{\ast},\qquad
f_{\ast} \in H_{0}^{2}(\mc{M};\R^{c}).
\end{equation}
Inserting $S_{\ast}$ and $S=g+h,\; h \in H_{0}^{1}(\mc{M};\R^{c})$, into \eqref{eq:VI-Sast} and partial integration gives
\begin{equation}\label{eq:VI-partially-integrated}
\la -\Delta S_{\ast}-\alpha S_{\ast},h-f_{\ast}\ra_{\mc{M}} \geq 0,
\end{equation}
where $\Delta S_{\ast} = (\Delta S_{\ast;1},\dotsc,\Delta S_{\ast;c})^{\T}$ applies componentwise.
Using the shorthands
\begin{subequations}
\begin{align}
\nu_{\alpha}(S_{\ast}) &= -\Delta S_{\ast}-\alpha S_{\ast},
\\
\mu_{\alpha}(S_{\ast}) &= \nu_{\alpha}(S_{\ast}) - \la\nu_{\alpha}(S_{\ast}),S_{\ast}\ra_{\R^{2}}\eins_{c},
\end{align}
\end{subequations}
where $\la\nu_{\alpha}(S_{\ast}),S_{\ast}\ra_{\R^{2}}$ denotes the function
$x \mapsto \big\la\nu_{\alpha}(S_{\ast})(x),S_{\ast}(x)\big\ra,\; x \in \mc{M}$,
we have
\begin{subequations}
\begin{align}\label{eq:mu-alpha-f-ast=0}
\la\mu_{\alpha}(S_{\ast}),S_{\ast}\ra_{\mc{M}} &= 0
\intertext{
since $\la \eins_{c},S_{\ast}(x)\ra=1$ for a.e.~$x$, and
}
\la \mu_{\alpha}(S_{\ast}),S\ra_{\mc{M}} &= \la\nu_{\alpha}(S_{\ast}),h-f_{\ast}\ra_{\mc{M}} \geq 0,
\end{align}
\end{subequations}
which is \eqref{eq:VI-partially-integrated}. Since $S(x) \geq 0$ a.e.~in $\mc{M}$ and may have arbitrary support, we deduce from the inequality $\la\mu_{\alpha}(S_{\ast}),S\ra_{\mc{M}} \geq 0$ and from the self-duality of the nonnegative orthant $\R_{+}^{c}$ that $\mu_{\alpha}(S_{\ast}) \geq 0$ a.e.~in $\mc{M}$. Since also $S_{\ast} \geq 0$ a.e., this implies that equation \eqref{eq:mu-alpha-f-ast=0} holds pointwise a.e.~in $\mc{M}$:
\begin{equation}
\mu_{\alpha}(S_{\ast})(x) S_{\ast}(x)
= \nu_{\alpha}(S_{\ast})(x)S_{\ast}(x) - \big\la\nu_{\alpha}(S_{\ast})(x) S_{\ast}(x)\big\ra S_{\ast}(x) = 0\qquad \text{a.e.~in}\;\mc{M}.
\end{equation}
Substituting $\nu_{\alpha}(S_{\ast})$ we deduce that a minimizer $S_{\ast}=g+f_{\ast}$ characterized by the variational inequality \eqref{eq:VI-Sast} weakly satisfies the PDE
\begin{equation}\label{eq:AF-PDE}
R_{S_{\ast}}(-\Delta S_{\ast}-\alpha S_{\ast}) = 0,
\end{equation}
where $R_{S_{\ast}}$ defined by \eqref{eq:def-Rp} applies $R_{S_{\ast}(x)}$ to vector $(-\Delta S_{\ast}-\alpha S_{\ast})(x)$ at every $x \in \mc{M}$.

\newpage
\begin{remark}[Comments]\label{rem:AF-PDE}$\text{ }$
\begin{enumerate}[(1)]
\item
We point out that computing a vector field $S_{\ast}$ satisfying \eqref{eq:VI-Sast} is difficult in practice, due to the nonconvexity of problem \eqref{eq:problem-for-numerics}. On the other hand, the algorithm proposed in Section \ref{sec:Numerical-Example} in the result illustrated by Figure \ref{fig:31-colors} shows that good suboptima can be computed by merely solving a sequence of simple problems.
\item
As already pointed out at the beginning of this section, the derivation of the PDE \eqref{eq:AF-PDE} is merely a formal one, due to the unrealistic regularity assumption \eqref{eq:S-regularity-assumption}.
In fact, since $\ker R_{S_{\ast}}(x)=\R\eins_{c}$, equation \eqref{eq:AF-PDE} says that $S_{\ast}$ is constant up to a set of measure zero. While the numerical result (Fig.~\ref{fig:31-colors}) clearly reflects this, the discontinuity of $S_{\ast}$ conflicts with assumption \eqref{eq:S-regularity-assumption}.
\end{enumerate}
\end{remark}

% !TEX root =  ../AF_and_Harmonic_Maps.tex
%%%%%%%%%%%%%%%%%%%%%%%%%%%%%%%%%%%%%%%%%%

\section{Conclusion}\label{sec:Conclusion}
We presented a novel parametrization of the assignment flow for contextual data classification on graphs. The dominating part of the flow admits the interpretation as Riemannian gradient flow with respect to the underlying information geometry, unlike the original formulation of the assignment flow. A decomposition of the corresponding potential by means of a non-local graph Laplacian makes explicit the interaction of two processes: regularization of label assignments and gradual enforcement of unambiguous decisions. The assignment flow combines these aspects in a seamless way, unlike traditional approaches where solutions to convex relaxations require postprocessing. It is remarkable that this behaviour is solely induced by the underlying information geometry.

We studied a continuous-domain variational formulation as counterpart of the discrete formulation restricted to a local discrete Laplacian (nearest neighbor interaction).
A numerical algorithm in terms of a sequence of simple linear elliptic problems reproduces results that were obtained with the original formulation of the assignment flow using completely different numerics (geometric ODE integration). This illustrates the derived mathematical relations.

\vspace{0.25cm}
We outline three attractive directions of further research.
\begin{itemize}
\item 
We clarified in Section \ref{sec:zero-scale-functional} that the inherent \textit{smooth} setting of the assignment flow \eqref{eq:assignment-flow} translates under suitable assumptions to the sequence of linear (discretized) elliptic PDE problems \eqref{eq:fk-PALM} together with a simple convex constraint. We did not touch on the limit problem, however. More mathematical work is required here, cf.~Remark \ref{rem:AF-PDE}.

Since the assignment flow returns image \textit{partitions} when applied to image features on a grid graph, the situation reminds us of the Mumford-Shah functional \cite{Mumford:1989aa} and its approximation by a sequence of $\Gamma$-converging smooth elliptic problems \cite{Ambrosio:1990aa}.
Likewise, one may regard the concave second term of \eqref{eq:problem-for-numerics} together with the convex constraint $S \in \mc{A}_{g}^{1}$ as a vector-valued counterpart of the basic nonnegative double-well potential of scalar phase-field models for binary segmentation \cite{Sternberg:1991aa,Cristoferi:2018aa}. In these works, too, nonsmooth limit cases result from $\Gamma$-converging simpler problems.
\item
Adopting the viewpoint of evolutionary dynamics \cite{Hofbauer:2003aa} on label assignment, the assignment flow may be characterized as spatially coupled replicator dynamics. To the best of our knowledge, our paper \cite{Astrom:2017ac} seems to be the first one that used information theory to formulate this spatial coupling. Some consequences of the geometry were elaborated in the present paper and discussed above. 

We point out that the literature on evolutionary dynamics in general, and specifically on the replicator equation, is vast. We merely point out few works on models involving the replicator equation and spatial interaction in physics 
\cite{Traulsen:2004aa,deForest:2013aa}, applied mathematics \cite{Novozhilov:2011aa,Bratus:2014aa},  including extensions to scenarios with an infinite number of strategies (as opposed to selecting from a finite set of labels) -- see \cite{Ambrosio:2018aa} and references therein.

In this context, our work might stimulate researchers working on spatially extended evolutionary dynamics in various scientific disciplines. In particular, generalizing our approach to continuous-domain integro-differential models seem attractive that conform to the assigment flow with \textit{non-local} interactions (i.e.~with larger neighborhoods $|\mc{N}_{i}|,\,i \in \mc{I}$) and the underlying geometry.
\item
Last but not least, our work may support a better understanding of \textit{learning} with networks. Our preliminary work on learning the weights \eqref{eq:weights-Omega-i} using the \textit{linearized} assignment flow \cite{Huhnerbein:2019aa} on a \textit{single} graph (`layer') revealed the model expressiveness of this limited scenario, on the one hand, and that subdividing complex learning tasks in this way avoids `black box behaviour', on the other hand. We hope that the continuous-domain perspective developed in this paper in terms of sequences of linear PDEs will support our further understanding of learning with hierarchical `deeper' architectures.
\end{itemize}

\begin{acknowledgments}
  Financial support by the German Science Foundation (DFG), grant GRK 1653, is gratefully acknowledged. This work has also been stimulated by the Heidelberg Excellence Cluster STRUCTURES, funded by the DFG under Germany's Excellen Strategy EXC-2181/1 - 390900948.
\end{acknowledgments}

%%
%\clearpage
%%%%%%%%%%
\bibliographystyle{amsalpha}
\bibliography{continuous-domain_assignment_flows}
\label{lastpage}
\end{document}